\documentclass{amsart}

\usepackage{graphicx,xcolor}
\usepackage{amssymb}
\usepackage{amscd}
\usepackage{longtable}
\usepackage{pb-diagram}
\usepackage[all]{xy}

\newtheorem{theorem}{Theorem}[section]
\newtheorem{proposition}[theorem]{Proposition}
\newtheorem{lemma}[theorem]{Lemma}

\newtheorem{problem}[theorem]{Problem}

\newtheorem{corollary}[theorem]{Corollary}
\newtheorem*{claim*}{Claim}
\theoremstyle{definition}

\newtheorem{example}[theorem]{Example}

\newcommand{\N}{\Bbb N}

\newcommand{\G}{G(K)}

\newcommand{\Z}{\Bbb Z}
\newcommand{\Q}{\Bbb Q}
\newcommand{\D}{\Delta}

\newcommand{\la}{\langle}
\newcommand{\ra}{\rangle}

\newcommand{\ord}{\mathcal{O}}

\theoremstyle{remark}
\newtheorem{remark}[theorem]{Remark}

\numberwithin{equation}{section}

\allowdisplaybreaks[4]

\begin{document}
\title{Twisted Alexander  vanishing groups of knots} 

\author{Katsumi Ishikawa, Takayuki Morifuji, and Masaaki Suzuki}

\thanks{2020 {\it Mathematics Subject Classification}. 
Primary 57K14, Secondary 57K10.}

\thanks{{\it Key words and phrases.\/}
Twisted Alexander polynomial, TAV group, TAV order.}

\begin{abstract}
In our previous work, we introduced the notion of a twisted Alexander vanishing (TAV) group, defined as a finite group for which the corresponding twisted Alexander polynomial of a knot vanishes. In this paper, we discuss the orders of TAV groups and construct knots whose twisted Alexander polynomials vanish. 
Moreover, we show that every faithful irreducible representation of a TAV group causes the twisted Alexander polynomial to be zero.  
\end{abstract}

\address{Research Institute for Mathematical Sciences, Kyoto University, Kyoto 606-8502, Japan}
\email{katsumi@kurims.kyoto-u.ac.jp}

\address{Department of Mathematics, Hiyoshi Campus, Keio University, Yokohama 223-8521, Japan}
\email{morifuji@keio.jp}

\address{Department of Frontier Media Science, 
Meiji University, 4-21-1 Nakano, Nakano-ku, Tokyo 
164-8525, Japan}
\email{mackysuzuki@meiji.ac.jp}

\maketitle

%%%%%%%%%%%%%%%%%%%%%%%%%%%%%%%%%%%%%%%%%%%%%%%%%%%%%%%%%%%%%%%%%%%%%%%%%%%%%%%%%%%%%%%%%%%%%%%%%%%%%%%%%%%%%%%%%%%%
\section{Introduction}

The twisted Alexander polynomial is introduced by Lin \cite{Lin01-1} for knots in the $3$-sphere $S^3$, and by Wada \cite{Wada94-1} for finitely presentable groups. It is a generalization of the classical Alexander polynomial, and is defined for the pair of a group (e.g., a knot group) and its representation. The theory of twisted Alexander polynomials has developed over the last twenty-five years. For example, it is known that fibered $3$-manifolds are detected by twisted Alexander polynomials associated with finite groups (see \cite{FV11-1}). See \cite{FV10-1}, \cite{Morifuji15-1}, \cite{Suzuki15-1} for other applications to knot theory and low-dimensional topology. 

We call a finite group $G$ a \textit{twisted Alexander vanishing (TAV) group of a knot $K$} in $S^3$ if there exists an epimorphism $f$ of the knot group $G(K)$ onto $G$ such that the twisted Alexander polynomial associated with  the composition of $f$ and the regular representation of $G$ is zero (see \cite{IMS23-1}). Furthermore, we call $K$ a \textit{fibered knot} if the complement of $K$ in $S^3$ admits a structure of a surface bundle over the circle such that the closures of the fiberes are Seifert surfaces. Then, a vanishing theorem for non-fibered knots due to Friedl and Vidussi is described as follows: 

\begin{theorem}[{\cite[Theorem~1.2]{FV13-1}}]\label{thm:FV}
Every non-fibered knot $K$ admits a TAV group of $K$. 
\end{theorem}

By abuse of terminology, we also call a finite group $G$ a \textit{TAV group} if $G$ is a TAV group of some knot $K$. 
In our previous paper \cite{IMS23-1}, we provide a characterization of a TAV group. In order to state the result precisely, 
we recall some terminologies in group theory. 
The \textit{weight} of a group $G$, denoted by $w (G)$, is the smallest integer $n$ such that $G$ is the normal closure of $n$ elements. In order to remove ambiguity, we set $w(\{e\})=0$. For a group $G$, there are a knot $K$ and an epimorphism $f\colon G(K)\to G$ if and only if $G$ is finitely generated and $w(G)\leq1$ (see \cite{GA75-1}, \cite{Johnson80-1}). 
We call an element $g$ of a group $G$ a {\it  weight element} if the normal closure of $g$ coincides with $G$. 
This term is found in \cite{Hillman}. 
For an example, a meridian of a knot $K$ is a weight element of the knot group $G(K)$. 
%Given a prime number $p$, a 
A finite group $G$ is a \textit{$p$-group} if and only if the order $|G|$ is a power of a prime number. 
%$p$. 
We simply denote the commutator subgroup $[G,G]$ of a group $G$ by $G'$. Then, we have:

\begin{theorem}[{\cite[Theorem~1.5]{IMS23-1}}]\label{thm:main-4}
A finite group $G$ is a TAV group if and only if 
$w(G)=1$ and $G'$ is not a $p$-group. 
\end{theorem}

For example, we easily see that the symmetric group $S_n~(n\geq4)$, the alternating group $A_n~(n\geq5)$, 
and the dihedral group $D_n$ of order $2n$ where $n$ is the product of at least two distinct odd primes are TAV groups. 
On the other hand, any abelian group is not a TAV group. 
Moreover, any $p$-group is not a TAV group. Since a finite $p$-group $G$ is nilpotent, $G$ is cyclic when $w(G)=1$. 

Let us recall the notion of the \textit{twisted Alexander vanishing order} 
$\ord(K)$ of a non-fibered knot $K$ (see \cite{IMS23-1} for details). 
We define $\ord(K)$ to be the order of the smallest TAV group of $K$, 
and call it the \textit{TAV order} of $K$ in short (we call it the \textit{minimal order} of $K$ in \cite{MS22-1}). 
The TAV orders of some knots are obtained in \cite[Theorem 3.2]{MS22-1} and \cite[Theorem 1.3]{IMS23-1}.
Here, we note that the smallest TAV group might be not unique for a non-fibered knot $K$ in general (see Corollary \ref{cor:273}). 
In \cite{IMS25-1}, we introduced the notion of a {\it seed} for a TAV group. 
If a TAV group is not a seed, then this group never realizes the TAV order of a knot, 
to be precise, there exists a smaller TAV group for which the twisted Alexander polynomial of a knot is zero (see Example \ref{exam:60210}). 

The purpose of this paper is to explore several properties of TAV groups. 
In particular, we discuss the following fundamental questions: 

\begin{itemize}
\item[(A)]
For a given natural number $n$, construct a TAV group of order $n$ and find the number of TAV groups of order $n$.

\item[(B)]
Find a non-fibered knot admitting a given TAV group.

\item[(C)]
Which irreducible representation of a TAV group provides the twisted Alexander polynomial zero. 

\end{itemize}

For (A), we consider the factorization of the order $n$ of a finite group. 
Suppose that $n$ is the product of at most four primes or square free.  
For such cases, the number of finite groups of order $n$ was determined in \cite{DEP}, \cite{Holder}. 
In this paper, we can select TAV groups among finite groups of order $n$. 

\begin{theorem}\label{thm:main-3}
Let $n$ be the product of at most four primes or square free.  
We can describe the number of TAV groups of order $n$. 
In particular, we can determine whether there exists a TAV group of order $n$ or not. 
\end{theorem}

The precise statement of Theorem \ref{thm:main-3} is given in 
Theorem \ref{thm_numTAVgroupsfourprimes} and Theorem \ref{numtavgroupsquarefree} as a partial answer for (A).   

By tracking the proof of Theorem \ref{thm:main-4}, we can theoretically find such a knot in (B), but it seems excessively impractical. As an accessible way to find a concrete non-fibered knot $K$ admitting a given TAV group $G$, we can adopt the \textit{satellite knot construction}. See Subsection \ref{subsec:3.0} for precise description of the knot. 
By \cite[Theorem 1.2]{IMS23-1}, it is easy to see that there are infinitely many \textit{composite} knots with the same TAV order. However, it is unclear that we can find infinitely many \textit{prime} knots with these TAV orders. 
As shown in \cite[Corollary 1.4]{IMS23-1}, there are infinitely many prime knots $K$ with $\ord(K)=24$. As for the question (B), we can extend the result in this paper as follows: 

\begin{theorem}\label{thm:main-5}
For any TAV group $G$, there are infinitely many hyperbolic knots that admit $G$ as a TAV group. Moreover, if $G$ is realized as a smallest TAV group of a knot, then infinitely many hyperbolic knots also admit $G$ as a smallest TAV group. 
\end{theorem}

Here, a knot $K$ in $S^3$ is \textit{hyperbolic} if the complement $S^3\setminus K$ admits the complete hyperbolic metric of finite volume. We remark that a hyperbolic knot is prime. 

In particular, if the order of a TAV group is the product of three distinct primes, we obtain the following. 

\begin{theorem}\label{thm:main-6}
Let $G$ be a TAV group of order $pqr$, where $p,q,r$ are primes.
Then there exist infinitely many hyperbolic knots that admit $G$ as a smallest TAV group.
%in particular, there exist infinitely many hyperbolic knots $K$ with $\ord(K) = pqr$.
\end{theorem}

It is well known that every representation of a finite group $G$ can be expressed as the direct sum of irreducible representations of $G$. 
In this sense, every irreducible representation of $G$ is an atom of representations of $G$.  
Even if the twisted Alexander polynomial associated to the regular representation is zero, 
it is unclear which irreducible representation causes the twisted Alexander polynomial to be zero. 
Then we will discuss the question (C) and obtain the following.  

\begin{theorem}\label{thm:main-7}
Let $G$ be a TAV group and $\rho \colon G \to {\rm GL}(V)$ a faithful irreducible representation of $G$. Then, there exist a knot $K$ and an epimorphism $f \colon G(K) \to G$ such that $\Delta_K^{\rho \circ f}(t) = 0$.
\end{theorem}

This paper is organized as follows. 
In Section \ref{sec:2}, 
we quickly recall the definition and one of vanishing criteria of twisted Alexander polynomials. 
In Section \ref{section:pqr}, we consider the above question (A). For a certain natural number $n$, we discuss TAV groups of order $n$.  
In Section \ref{sec:4}, we construct a knot admitting a given TAV group as an answer for (B), and prove Theorems \ref{thm:main-5} and \ref{thm:main-6}. 
In Section \ref{sec:5}, we give a proof of Theorem \ref{thm:main-7} as an answer for (C). 

A part of this paper is contained in our unpublished manuscript \cite{IMS-II}. In this paper, we develop it from the viewpoint of TAV groups. 

%%%%%%%%%%%%%%%%%%%%%%%%%%%%%%%%%%%%%%%%%%%%%%%%%%%%%%%%%%%%%%%%%%%%%%%%%%%%%%%%%%%%%%%%%%%%%%%%%%%%%%%%%%%%%%%%%%%%
\section{Preliminaries}\label{sec:2}
%%%%%%%%%%%%%%%%%%%%%%%%%%%%%%%%%%%%%%%

Let $X$ be a connected finite CW complex, 
$\phi\in H^1(X;\Z)=\mathrm{Hom}(\pi_1(X),\Z)$, 
and $\rho\colon\pi_1(X)\to \mathrm{GL}(n, R)$ a homomorphism to a general linear group over a Noetherian unique factorization domain $R$. 
Define a right $\Z[\pi_1(X)]$-module structure on 
$R^n\otimes_\Z\Z[t^{\pm1}]=R[t^{\pm1}]^n$ as follows: 
$$
(v\otimes p)\cdot g=(v\cdot \rho(g))\otimes(p\cdot t^{\phi(g)}),
$$
where 
$g\in \pi_1(X)$ and $v\otimes p\in R^n\otimes_\Z\Z[t^{\pm1}]$. 
Here, we view $R^n$ as row vectors. 
Taking tensor product, we obtain a homomorphism 
$\rho\otimes\phi\colon\pi_1(X)\to \mathrm{GL}(n,R[t^{\pm1}])$. 

We denote by $\tilde{X}$ the universal covering of $X$, 
and use the homomorphism 
$\rho\otimes\phi$ to regard 
$R[t^{\pm1}]^n$ as a right $\Z[\pi_1(X)]$-module. 
The chain complex $C_*(\tilde{X})$ is a left 
$\Z[\pi_1(X)]$-module via deck transformations. 
We can therefore consider the tensor products
$$
C_*(X;R[t^{\pm1}]^n)
:=R[t^{\pm1}]^n\otimes_{\Z[\pi_1(X)]}C_*(\tilde{X}),
$$ 
which form a chain complex of $R[t^{\pm1}]$-modules. 
We then consider the $R[t^{\pm1}]$-modules 
$H_*(X;R[t^{\pm1}]^n)
:=H_*(C_*(X;R[t^{\pm1}]^n))$. 

Since $X$ is compact and $R[t^{\pm1}]$ is Noetherian, 
these modules are finitely presented over $R[t^{\pm1}]$. 
We then define the \textit{twisted Alexander polynomial} of $(X,\phi,\rho)$ 
to be the order of $H_1(X;R[t^{\pm1}]^n)$ 
as a left $R[t^{\pm1}]$-module. 
We will denote it as $\D_{X,\phi}^\rho(t)\in R[t^{\pm1}]$, 
and note that $\D_{X,\phi}^\rho(t)$ is well defined up to multiplication by a unit in $R[t^{\pm1}]$. See \cite{FV11-1} for other basic properties of twisted Alexander polynomials. 

For a homomorphism $f \colon \pi_1(X) \to G$ to a finite group $G$, we get the representation
$$
\pi_1(X)\overset{f}{\longrightarrow} G
\overset{\rho}{\longrightarrow} \mathrm{Aut}_\Z(\Z[G]),
$$ 
where the second map is given by the right multiplication. 
We can also identify $\mathrm{Aut}_\Z(\Z[G])$ with 
$\mathrm{GL}(|G|,\Z)$, and 
obtain the corresponding twisted Alexander polynomial $\Delta_{X, \phi}^{\rho \circ f}(t)$. Then, the vanishing of $\Delta_{X, \phi}^{\rho \circ f}(t)$ is characterized as follows:

\begin{theorem}[{\cite[Theorem~4.2]{IMS23-1}}]\label{lifting-thm}
The twisted Alexander polynomial $\Delta^{\rho \circ f}_{X, \phi}(t)$ is zero if and only if there exists a nontrivial lift $\tilde{f} \colon \pi_1(X) \to \mathbb{Z}[G \times \mathbb{Z}] \rtimes (G \times \mathbb{Z})$, where $G \times \mathbb{Z}$ acts on $\mathbb{Z}[G \times \mathbb{Z}]$ by the left multiplication, of the homomorphism $f \times \phi \colon \pi_1(X) \to G \times \mathbb{Z}$, i.e., a group homomorphism $\tilde{f}$ such that $p_{G \times \mathbb{Z}} \circ \tilde{f} = f \times \phi$ and ${\rm Im}\; \tilde{f} \cap (\mathbb{Z}[G \times \mathbb{Z}] \times \{(e, 0)\}) \neq \{(0; e, 0)\}$, where $p_{G \times \mathbb{Z}} \colon \mathbb{Z}[G \times \mathbb{Z}] \rtimes (G \times \mathbb{Z}) \to G \times \mathbb{Z}$ is the projection.
\end{theorem}

In this paper, we consider a knot $K$ in the $3$-sphere $S^3$ 
and let $E_K=S^3\setminus \nu(K)$, where $\nu(K)$ denotes an open tubular neighborhood of $K$. 
We denote $\pi_1(E_K)$ by $G(K)$, 
and call it the \textit{knot group} of $K$. 
If $f\colon\G\to G$ is an epimorphism to a finite group $G$, we obtain the twisted Alexander polynomial 
$\D_{E_K,\phi}^{\rho\circ f}(t)$. 
For the \textit{abelianization} homomorphism 
$\phi \colon G(K)\to H_1(E_K;\Z)\cong\Z$, 
we drop $\phi$ from the notation and use $\D_K^{\rho\circ f}(t)$ for simplicity. 

\begin{remark}\label{rmk:torsion}
It is known that 
$\D_{K}^{\rho\circ f}(t)\not=0$ if and only if 
$H_1(E_K;\Q[G][t^{\pm1}])=H_1(E_K;\Q[t^{\pm1}]^{|G|})$ is $\Q[t^{\pm1}]$-torsion, 
namely, 
$\mathrm{rank}_\Z\,H_1(E_K;\Z[t^{\pm1}]^{|G|})$ is finite 
(see \cite[Remark~4.5]{Turaev01-1}). 
\end{remark}

%%%%%%%%%%%%%%%%%%%%%%%%%%%%%%%%%%%%%%%%%%%%%%%%%%%%%%%%%%%%%%%%%%%%%%%%%%%%%%%%%%%%%%%%%%%%%%%%%%%%%%%%%%%%%%%%%%%%

%%%%%%%%%%%%%%%%%%%%%%%%%%%%%%%%%%%%%%%%%%%%%%%%%%%%%%%%%%%%%%%%%%%%%%%%%%%%%%%%%%%%%%%%%%%%%%%%%%%%%%%%%%%%%%%%%%%%
\section{TAV groups of order the product of primes}\label{section:pqr}

In this section, we consider TAV groups whose orders are the products of at most four primes or square free. 

\subsection{TAV groups of order the product of at most four primes}

For natural numbers $u$ and $v$, we define 
\[
D_u^v =
\left\{
\begin{array}{ll} 
1 & \mbox{ if }u \mbox{ divides } v, \\
0 & \mbox{ otherwise} 
\end{array} .
\right. 
\]
Remark that $D_u^v$ is defined as $\Delta_v^u$ in \cite{DEP}. 

\begin{theorem}\label{thm_numTAVgroupsfourprimes}
Let $G$ be a TAV group whose order factorizes into at most four primes. 
Then the order of $G$ is $pqr$, $p^3q$, $p^2qr$, or $pqrs$, where $p,q,r,s$ are distinct primes. 
More precisely, we obtain the following. 
\begin{itemize}
\item[(i)] 
The number of TAV groups of order $pqr$, where $p< q < r$, is 
\[
(p-1) D_p^{q-1} D_p^{r-1} .
\]  
\item[(ii)] 
The TAV group of order $p^3 q$ is only $S_4$.%, where $p=2,q=3$. 
\item[(iii)] 
The number of TAV groups of order $p^2qr$, where $q < r$, is 
\begin{align*}
& D_p^{r - 1} D_q^{p - 1} 
+ (q - 1) D_q^{r - 1} D_q^{p - 1} 
+ (q - 1) D_p^{r - 1} D_q^{r - 1} D_q^{p - 1} \\
&+ (p - 1) D_p^{q - 1} D_p^{r - 1} 
+(p - 1) D_{p^2}^{q - 1} D_p^{r - 1} 
+ (p - 1) D_p^{q - 1} D_{p^2}^{r - 1} 
+ (p^2 - p) D_{p^2}^{q - 1} D_{p^2}^{r - 1} \\
&+ (q - 1) D_q^{p - 1} D_q^{r - 1} \\
& + D_2^q  
+ \frac{1}{2} q (q - 1 - D_2^q) D_q^{p - 1} D_q^{r - 1} 
+ \frac{1}{2} (q - 1 - D_2^q) D_q^{p + 1} D_q^{r - 1} \\
& + D_2^q D_r^{p - 1} 
+ D_2^q D_r^{p + 1} \\
&+ D_2^p D_3^q D_5^r .
\end{align*}
\item[(iv)]  
The number of TAV groups of order $pqrs$, where $p<q<r<s$, is 
\begin{align*}
&(p - 1) D_p^{q - 1} D_p^{r - 1} \\
&+ (p - 1) D_p^{q - 1} D_p^{s - 1} 
+ D_p^{q - 1} D_r^{s - 1} 
+ (p - 1) D_p^{q - 1} D_p^{s - 1} D_r^{s - 1} \\
& + D_p^{r - 1} D_q^{s - 1} 
+ D_p^{s - 1} D_q^{r - 1} 
+ (p - 1) D_p^{r - 1} D_p^{s - 1} 
+ (q - 1) D_q^{r - 1} D_q^{s - 1} \\
&+ (p - 1) D_p^{r - 1} D_p^{s - 1} D_q^{r - 1} 
+ (p - 1) D_p^{r - 1} D_p^{s - 1} D_q^{s - 1} \\
&+ (q - 1) D_q^{r - 1} D_q^{s - 1} D_p^{r - 1} 
+ (q - 1) D_q^{r - 1} D_q^{s - 1} D_p^{s - 1} \\
&+ (p q - p - q + 1) D_p^{r - 1} D_p^{s - 1} D_q^{r - 1} D_q^{s - 1} \\
&+ (p - 1)^2 D_p^{q - 1} D_p^{r - 1} D_p^{s - 1} .
\end{align*}
\end{itemize}
\end{theorem}

\begin{proof}
Let $G$ be a group of order $n$. If $n$ is a power of a prime $p$, 
then the commutator subgroup $G'$ is also a $p$-group. 
If $n$ is the product $pq$ of two distinct primes $p$ and $q$, then $G$ is solvable, which means that the order of $G'$ is $1, p$, or $q$. 
Thus, $G$ is not a TAV group in these cases.
If $n = p^2 q$ for distinct primes $p$ and $q$, it is known (e.g., \cite[Theorem 1.31]{Isaacs08-1}) that a group $G$ of order $n$ has a normal Sylow $p$-subgroup or a normal Sylow $q$-subgroup. Let $N$ be a normal Sylow subgroup. Since $|G / N|$ equals $q^2$ or $p$, $G/N$ is abelian and hence $G' \subset N$. Thus, $G'$ has prime-power order and $G$ is not a TAV group.
Any finite group $G$ of order $p^2 q^2$, where $p<q$ (except when $n = 36$) has a normal Sylow $q$-subgroup (see \cite[1E.1]{Isaacs08-1}). 
Then, as in the case $n = p^2 q$, $G$ is not a TAV group, except for $n =36$. 
Moreover, we see any finite group of order $36$ is not a TAV group by the previous papers \cite{MS22-1}, \cite{IMS23-1}.

(i) 
Let us construct TAV groups of order $p q r$, where $p,q,r$ are distinct primes with $p < q < r$. We assume $q \equiv r \equiv 1 \;(\text{mod $p$})$ and take integers $a$ and $b$ that define elements of order $p$ in the multiplicative groups $(\mathbb{F}_q)^\times$ and $(\mathbb{F}_r)^\times$ respectively, 
where $\mathbb{F}_q$ and $\mathbb{F}_r$ are finite fields of order $q$ and $r$ respectively. 
We assume $a^p \equiv 1 \;(\text{mod $q$})$, $a \not\equiv 1 \;(\text{mod $q$})$, $b^p \equiv 1 \;(\text{mod $r$})$, and $b \not\equiv 1 \;(\text{mod $r$})$. 
The $a$-th and the $b$-th powers define automorphisms on the cyclic groups $C_q$ and $C_r$, and let $G(pqr;a,b)$ be the semidirect product $(C_q \times C_r) \rtimes C_p$ with respect to these automorphisms, i.e.,
$$G(pqr;a,b) = \langle x, y, z \mid x^q, y^r, z^p, xyx^{-1}y^{-1}, zxz^{-1}x^{-a}, zyz^{-1}y^{-b} \rangle.$$
We can easily check that $G(pqr;a, b)$ is a TAV group, because its commutator subgroup is $C_q\times C_r\cong C_{qr}$.

If $q \equiv r \equiv 1 \;(\text{mod $p$})$, there exists a TAV group $G(pqr;a,b)$ of order $pqr$ as seen above. To see the converse, recall that any group with square-free order is metacyclic (e.g., see \cite[Theorem 5.16]{Isaacs08-1}). Therefore, if $G$ is a TAV group of order $pqr$, $G/G' \cong C_{p'}$ and $G' \cong C_{q' r'} \cong C_{q'} \times C_{r'}$, where $\{ p, q, r \} = \{p', q', r' \}$. Since the actions of $C_{p'}$ on $C_{q'}$ and $C_{r'}$ are nontrivial, 
we have $q' \equiv r' \equiv 1 \;(\text{mod $p'$})$; in particular, $p' = p$ and this concludes the converse. Furthermore, this proof also shows that $G \cong G(pqr; a, b)$ for some $a$ and $b$.

Let us assume $q \equiv r \equiv 1 \;(\text{mod $p$})$. For $i = 1, \dots, p-1$ and any integers $a$ and $b$ having order $p$ in $(\mathbb{F}_q)^\times$ and $(\mathbb{F}_r)^\times$, we can define an isomorphism from $G(pqr; a^i, b^i)$ to $G(pqr; a, b)$ by
$$x \mapsto x, \quad y \mapsto y, \quad z \mapsto z^i.$$
Therefore any TAV group of order $pqr$ is isomorphic to one of
$$G(pqr; a, b), G(pqr; a, b^2), \dots, G(pqr; a, b^{p-1})$$
for fixed $a, b$.

Finally, we show that the $p - 1$ TAV groups $G(pqr; a, b^i)$ are not isomorphic. Let $\varphi\colon G(pqr; a, b^i) \to G(pqr; a, b^j)$ be an isomorphism. For any $\alpha \in (\mathbb{F}_q)^\times$ and $\beta \in (\mathbb{F}_r)^\times$,
$$x \mapsto x^\alpha, \quad y \mapsto y^\beta, \quad z \mapsto z$$
define an automorphism on $G(pqr; a, b^j)$. Therefore, by composing such an automorphism with $\varphi$, we may assume $\varphi (x) = x, \varphi(y) = y.$ Taking a conjugation by some element of $(G(pqr; a, b^j))'$, we can further assume that $\varphi(z) = z^k$ for some $k$. Since $\varphi(z) \varphi(x) \varphi(z)^{-1} \varphi(x)^{-a} = x^{a^k - a}$ and $\varphi(z) \varphi(y) \varphi(z)^{-1} \varphi(y)^{-b^i} = y^{b^{jk} - b^i}$, we have $k \equiv 1 \; (\text{mod $p$})$ and $jk \equiv i \;(\text{mod $p$})$, which concludes that $i = j$.

(ii) Any finite group $G$ of order $p^3 q$ (except when $n = 24$) 
has either a normal Sylow $p$-subgroup or a normal Sylow $q$-subgroup (see \cite[Theorem 1.32]{Isaacs08-1}). 
Then $G$ is not a TAV group, except for $n =24$. 
Moreover, we see only $S_4$ is a TAV group of order $24$ by the previous papers \cite{MS22-1}, \cite{IMS23-1}.

(iii) Dietrich, Eick, and Pan in \cite{DEP} show the explicit presentations of finite groups of order $p^2 q r$. 
Then we obtain all the following tables of TAV groups of $p^2 q r$, where $q < r$ . 
Note that there exists only one non-solvable group $A_5$ of order $p^2 q r$. 

\begin{tabular}{|l|l|}
\hline
$| \{C_{pr} \rtimes C_{pq} \} |$ & conditions \\
\hline
$1$ & $p-1 \equiv 0 \pmod q ,  r-1 \equiv 0 \pmod p ,  r-1 \not\equiv 0 \pmod q$ \\
$q-1$ & $p-1 \equiv 0 \pmod q ,  r-1 \not\equiv 0 \pmod p ,  r-1 \equiv 0 \pmod q$ \\
$2q-1$ & $p-1 \equiv 0 \pmod q ,  r-1 \equiv 0 \pmod p ,  r-1 \equiv 0 \pmod q$ \\
\hline
\end{tabular}

\begin{tabular}{|l|l|}
\hline
$| \{C_{qr} \rtimes C_{p^2} \} |$ & conditions \\
\hline
$p-1$ & $q-1 \equiv 0 \pmod p ,  q-1 \not\equiv 0 \pmod {p^2} ,  r-1 \equiv 0 \pmod p ,  r-1 \not\equiv 0 \pmod {p^2}$ \\
$2p-2$ & $q-1 \equiv 0 \pmod {p^2} ,  r-1 \equiv 0 \pmod p ,  r-1 \not\equiv 0 \pmod {p^2}$ \\
$2p-2$ & $q-1 \equiv 0 \pmod p ,  q-1 \not\equiv 0 \pmod {p^2} ,  r-1 \equiv 0 \pmod {p^2}$ \\
$p^2+2p-3$ & $q-1 \equiv 0 \pmod {p^2} ,  r-1 \equiv 0 \pmod {p^2}$ \\
\hline
\end{tabular}

\begin{tabular}{|l|l|}
\hline
$| \{C_{p^2r} \rtimes C_{q} \} |$ & conditions \\
\hline
$q-1$ & $p-1 \equiv 0 \pmod q ,  r-1 \equiv 0 \pmod q$ \\
\hline
\end{tabular}

\begin{tabular}{|l|l|}
\hline
$| \{(C_{pr} \times C_p) \rtimes C_{q} \} |$ & conditions \\
\hline
$p-1$ & $q=2$ \\
$\frac{1}{2}q (q-1)$ & $q \neq 2 , p-1 \equiv 0 \pmod q ,  r-1 \equiv 0 \pmod q$ \\
$\frac{1}{2}(q-1)$ & $q \neq 2 , p+1 \equiv 0 \pmod q ,  r-1 \equiv 0 \pmod q$ \\
\hline
\end{tabular}

\begin{tabular}{|l|l|}
\hline
$| \{((C_{p} \times C_p) \rtimes C_{r}) \rtimes C_q \} |$ & conditions \\
\hline
$1$ & $q = 2 , p-1 \equiv 0 \pmod r$ \\
$1$ & $q = 2 , p+1 \equiv 0 \pmod r$ \\
\hline
\end{tabular}

\begin{tabular}{|l|l|}
\hline
$| \{A_5 \} |$ & conditions \\
\hline
$1$ & $p = 2, q = 3 , r = 5$ \\
\hline
\end{tabular}

(iv) Finally, we discuss TAV groups of order $pqrs$, where $p<q<r<s$. 
In general, a finite group of order square-free is metacyclic and then can be obtained by 
\[
(C_{p_1} \times  C_{p_2} \times \cdots \times C_{p_n}) \rtimes H
\]
where $H$ is a subgroup of ${\rm Aut} (C_{p_1} \times  C_{p_2} \times \cdots \times C_{p_n})$. 
Moreover, different subgroups of ${\rm Aut} (C_{p_1} \times  C_{p_2} \times \cdots \times C_{p_n})$ give rise to distinct finite groups 
(c.f. \cite[Section 21.3]{BNV}, \cite[Theorem 9.4.3]{Hall}). 
Since the commutator subgroup of a TAV group is not a $p$-group and finite simple groups of order square-free are abelian, 
the commutator subgroup of a TAV group is $C_a \times C_b \simeq C_{ab}$ or $C_a \times C_b \times C_c \simeq C_{abc}$, where $a,b,c \in \{p,q,r,s\}$.  
Then all TAV groups $G$ of order $pqrs$ are expressed as 
\[
G \simeq (C_a \times C_b) \rtimes H \mbox{ or } (C_a \times C_b \times C_c) \rtimes H,
\]  
where $H$ is  a subgroup of ${\rm Aut}(C_a \times C_b)$ or ${\rm Aut}(C_a \times C_b \times C_c)$ respectively. 
%and the order of $H$ is $pqrs/ab$ or $pqrs/abc$ respectively.  
If $H$ acts trivially on each factor of $C_a \times C_b$ (respectively $C_a \times C_b \times C_c$), 
then the commutator subgroup is not $C_a \times C_b$ (respectively $C_a \times C_b \times C_c$). 

Case (1): $G' = C_{qr}$ and $G/G' = C_{ps}$. \\
Since $H$ is a subgroup of ${\rm Aut}(C_q \times C_r)  \simeq C_{q-1} \times C_{r-1}$ and $q, r < s$, 
the order of $H$ is $p$ and $q-1 \equiv 0 \pmod p$, $r-1 \equiv 0 \pmod p$. 
Moreover, $H$ acts nontrivially on both $C_q$ and $C_r$. 
Then $H (< C_{q-1} \times C_{r-1} \simeq {\mathbb Z}/(q-1) {\mathbb Z} \times {\mathbb Z}/(r-1) {\mathbb Z})$ is expressed as   
\[
\left\langle \left( \frac{q-1}{p}, \frac{r-1}{p} \right) \right\rangle, 
\left\langle \left( \frac{q-1}{p}, 2 \frac{r-1}{p} \right) \right\rangle, 
\ldots  , 
\left\langle \left( \frac{q-1}{p}, (p-1) \frac{r-1}{p} \right) \right\rangle .
\]
Therefore there are $p-1$ TAV groups in this case, which are 
\[
G = ((C_q \times C_r) \rtimes C_p ) \times C_s \simeq C_{qr} \rtimes C_{ps}. 
\]

Case (2): $G' = C_{qs}$ and $G/G' = C_{pr}$. \\
If $q-1 \equiv 0 \pmod p,  s-1 \equiv 0 \pmod r$, then $H < {\rm Aut}(C_q \times C_s)$ of order $pr$ is expressed as   
\[
\left\langle \left( \frac{q-1}{p}, \frac{s-1}{r} \right) \right\rangle 
\]
and the TAV group is 
\[
G = (C_q \rtimes C_p) \times (C_s \rtimes C_r ) \simeq C_{qs} \rtimes C_{pr}. 
\]
If $q-1 \equiv 0 \pmod p,  s-1 \equiv 0 \pmod p$, then $H < {\rm Aut}(C_q \times C_s)$ of order $p$, 
which is similar to Case (1), is expressed as   
\[
\left\langle \left( \frac{q-1}{p}, \frac{s-1}{p} \right) \right\rangle, 
\left\langle \left( \frac{q-1}{p}, 2 \frac{s-1}{p} \right) \right\rangle, 
\ldots  , 
\left\langle \left( \frac{q-1}{p}, (p-1) \frac{s-1}{p} \right) \right\rangle .
\]
and the TAV groups are  
\[
G = ((C_q \times C_s) \rtimes C_p ) \times C_r \simeq C_{qs} \rtimes C_{pr}. 
\]
Furthermore, if $q-1 \equiv 0 \pmod p ,  s-1 \equiv 0 \pmod p ,  s-1 \equiv 0 \pmod r$, 
$H$ of order $p$ or $pr$ is expressed as 
\begin{align*}
& \left\langle \left( \frac{q-1}{p}, \frac{s-1}{r} \right) \right\rangle , \\
& \left\langle \left( \frac{q-1}{p}, \frac{s-1}{p} \right) \right\rangle, 
\left\langle \left( \frac{q-1}{p}, 2 \frac{s-1}{p} \right) \right\rangle, 
\ldots  , 
\left\langle \left( \frac{q-1}{p}, (p-1) \frac{s-1}{p} \right) \right\rangle, \\
&\left\langle \left( \frac{q-1}{p}, \frac{s-1}{pr} \right) \right\rangle, 
\left\langle \left( \frac{q-1}{p}, 2 \frac{s-1}{pr} \right) \right\rangle, 
\ldots  , 
\left\langle \left( \frac{q-1}{p}, (p-1) \frac{s-1}{pr} \right) \right\rangle .
\end{align*}
It is easy to see that the above subgroups are mutually different. 
In the last $p-1$ cases,  the TAV groups are 
\[
G = (C_q \times C_s) \rtimes (C_p \times C_r) \simeq C_{qs} \rtimes C_{pr}. 
\]

Case (3): $G' = C_{rs}$ and $G/G' = C_{pq}$. \\
We see only the case $r-1 \equiv 0 \pmod p ,  s-1 \equiv 0 \pmod p , r-1 \equiv 0 \pmod q ,  s-1 \equiv 0 \pmod q $. 
The other cases are the same as we mentioned in the above cases. 
In this case, the order of a subgroup $H$ of ${\rm Aut}(C_r \times C_s)$ is $p$, $q$, or $pq$. 
If the order of $H$ is $p$ or $q$, there are $p-1$ or $q-1$ TAV groups 
\[
G = ((C_r \times C_s) \rtimes C_p ) \times C_q \simeq C_{rs} \rtimes C_{pq} \mbox{ or } ((C_r \times C_s) \rtimes C_q ) \times C_p \simeq C_{rs} \rtimes C_{pq} 
\]
respectively, which are similar to Case(1), (2). 
If the order of $H$ is $pq$,  $H < {\rm Aut}(C_r \times C_s)$ is expressed as 
\[
\left\langle \left( k \frac{r-1}{pq}, \frac{s-1}{pq} \right) \right\rangle, 
\left\langle \left( \frac{r-1}{p}, l \frac{s-1}{pq} \right) \right\rangle, 
\left\langle \left( \frac{r-1}{p}, \frac{s-1}{q} \right) \right\rangle, 
\left\langle \left( \frac{r-1}{q}, \frac{s-1}{p} \right) \right\rangle .
\]
where $1 \leq k,l \leq pq - 1$. However, if $kl \equiv 1 \pmod {pq}$, then 
\[
\left\langle \left( k \frac{r-1}{pq}, \frac{s-1}{pq} \right) \right\rangle =  
\left\langle \left( \frac{r-1}{p}, l \frac{s-1}{pq} \right) \right\rangle. 
\]
Then the number of subgroups $H$ of order $pq$ is 
\[
(pq-1) + (pq-1) + 2 - (p-1)(q-1) = pq + p +q-1.
\]
Therefore we have $ (p-1) + (q-1) + (pq+p+q-1) =pq + 2 p + 2q -3$ TAV groups in this case. 

Case (4): $G' = C_{qrs}$ and $G/G' = C_{p}$. \\
Since $H$ is a subgroup of ${\rm Aut}(C_q \times C_r \times C_s) \simeq C_{q-1} \times C_{r-1} \times C_{s-1}$, 
the order of $H$ is $p$ and $q-1 \equiv 0 \pmod p$, $r-1 \equiv 0 \pmod p$, $s-1 \equiv 0 \pmod p$. 
Then $H (< C_{q-1} \times C_{r-1} \times C_{s-1})$ is expressed as   
\[
\left\langle \left( \frac{q-1}{p}, k \frac{r-1}{p}, l \frac{s-1}{p} \right) \right\rangle 
\]
where $1 \leq k,l \leq p - 1$.
Therefore, in this case we have $(p-1)^2$ TAV groups 
\[
G = (C_q \times C_r \times C_s ) \rtimes C_p \simeq C_{qrs} \rtimes C_p. 
\]
As a consequence, we have the following tables and this completes the proof. 

\begin{tabular}{|l|l|}
\hline
$| \{C_{qr} \rtimes C_{ps} \} |$ & conditions \\
\hline
$p-1$ & $q-1 \equiv 0 \pmod p ,  r-1 \equiv 0 \pmod p$ \\
\hline
\end{tabular}

\begin{tabular}{|l|l|}
\hline
$| \{C_{qs} \rtimes C_{pr} \} |$ & conditions \\
\hline
$1$ & $q-1 \equiv 0 \pmod p ,  s-1 \not\equiv 0 \pmod p ,  s-1 \equiv 0 \pmod r$ \\
$p-1$ & $q-1 \equiv 0 \pmod p ,  s-1 \equiv 0 \pmod p ,  s-1 \not\equiv 0 \pmod r$ \\
$2p-1$ & $q-1 \equiv 0 \pmod p ,  s-1 \equiv 0 \pmod p ,  s-1 \equiv 0 \pmod r$ \\
\hline
\end{tabular}

\begin{tabular}{|l|l|}
\hline
$| \{C_{rs} \rtimes C_{pq} \} |$ & conditions \\
\hline
$1$ & $r-1 \equiv 0 \pmod p ,  s-1 \not\equiv 0 \pmod p , r-1 \not\equiv 0 \pmod q ,  s-1 \equiv 0 \pmod q $ \\
$1$ & $r-1 \not\equiv 0 \pmod p ,  s-1 \equiv 0 \pmod p , r-1 \equiv 0 \pmod q ,  s-1 \not\equiv 0 \pmod q $ \\
$p-1$ & $r-1 \equiv 0 \pmod p ,  s-1 \equiv 0 \pmod p , r-1 \not\equiv 0 \pmod q ,  s-1 \not\equiv 0 \pmod q $ \\
$q-1$ & $r-1 \not\equiv 0 \pmod p ,  s-1 \not\equiv 0 \pmod p , r-1 \equiv 0 \pmod q ,  s-1 \equiv 0 \pmod q $ \\
$2p-1$ & $r-1 \equiv 0 \pmod p ,  s-1 \equiv 0 \pmod p , r-1 \equiv 0 \pmod q ,  s-1 \not\equiv 0 \pmod q $ \\
$2p-1$ & $r-1 \equiv 0 \pmod p ,  s-1 \equiv 0 \pmod p , r-1 \not\equiv 0 \pmod q ,  s-1 \equiv 0 \pmod q $ \\
$2q-1$ & $r-1 \equiv 0 \pmod p ,  s-1 \not\equiv 0 \pmod p , r-1 \equiv 0 \pmod q ,  s-1 \equiv 0 \pmod q $ \\
$2q-1$ & $r-1 \not\equiv 0 \pmod p ,  s-1 \equiv 0 \pmod p , r-1 \equiv 0 \pmod q ,  s-1 \equiv 0 \pmod q $ \\
$pq+2p+2q-3$ & $r-1 \equiv 0 \pmod p ,  s-1 \equiv 0 \pmod p , r-1 \equiv 0 \pmod q ,  s-1 \equiv 0 \pmod q $ \\
\hline
\end{tabular}

\begin{tabular}{|l|l|}
\hline
$| \{C_{qrs} \rtimes C_{p} \} |$ & conditions \\
\hline
$(p-1)^2$ & $q-1 \equiv 0 \pmod p ,  r-1 \equiv 0 \pmod p ,  s-1 \equiv 0 \pmod p$ \\
\hline
\end{tabular}

\end{proof}

\begin{example}\label{exam:60210}
Firstly, we consider finite groups of order $60 = 2^2 \cdot 3 \cdot 5$. 
In this case, there are three TAV groups  ${\rm Dic}_{15}$, $A_5$, $C_3 \rtimes F_5$.  
Since their commutator subgroups are $C_{15}, A_5, C_{15}$ respectively and they are not $p$-groups. 
We have a criterion to determine whether a finite group is a seed or not, by \cite{IMS25-1}. 
Namely, a finite group $G$ of weight one is not a seed if and only if there exists a nontrivial cyclic group $C$ in the center of $G$ such that $C \cap G' = \{e\}$. 
Then $A_5$ and $C_3 \rtimes F_5$ are seeds, since their centers are trivial.  
On the other hand, the dicyclic group ${\rm Dic}_{15}$ is not a seed, 
since the center is $C_2$ which is not contained in $[{\rm Dic}_{15}, {\rm Dic}_{15}] = C_{15}$. 
It follows that if ${\rm Dic}_{15}$ induces the twisted Alexander polynomial of a knot $K$ to be zero, 
then so does the dihedral group $D_{15}$ (cf. \cite{IMS25-1}). Therefore the TAV order of $K$ is smaller than or equal to $30$. 

Secondly, we consider finite groups of order $210 = 2 \cdot 3 \cdot 5 \cdot 7$. 
In this case, there are six TAV groups 
$D_5 \times (C_7 \rtimes C_3)$, $C_5 \rtimes F_7$, $C_3 \times D_{35}$, 
$C_5 \times D_{21}$, $C_7 \times C_{15}$, $D_{105}$. 
Their commutator subgroups are $C_{35}, C_{35}, C_{35}, C_{21}, C_{15}, C_{105}$ respectively. 
%They are also described as $C_{35} \rtimes C_{6}$, $C_{35} \rtimes C_{6}$, $C_{35} \rtimes C_{6}$, 
%$C_{15} \rtimes C_{14}, C_{21} \rtimes C_{10}, C_{105} \rtimes C_{2}$ respectively. 
The centers of $D_5 \times (C_7 \rtimes C_3)$, $C_5 \rtimes F_7$, $D_{105}$ are trivial, then they are seeds. 
While the others are not seeds. 
\end{example}

\subsection{TAV groups of order square free}

In general, we can determine the number of TAV groups of order square free. 
For $m \in {\mathbb N}$, let $\pi (m)$ be the set of primes dividing $m$. 
We define some complementary functions as follows. 
For $m \in {\mathbb N}$ and a prime $p$, 
\begin{align*}
c_m(p) &= | \{ q \in \pi (m) \, | \, q - 1 \equiv 0 \pmod p \}| , \\
\bar{c}_m(p) &= | \{ q \in \pi (m) \, | \, p - 1 \equiv 0 \pmod q \}| .  
\end{align*}
H\"{o}lder in \cite{Holder} determined the number of finite groups of order square free. 

\begin{theorem}[H\"{o}lder \cite{Holder}]\label{holderthm}
Let $n \in {\mathbb N}$ be square free. 
The number of groups of order $n$ is
\[
\sum_{m|n} \, \prod_{p \in \pi(\frac{n}{m})} \frac{p^{c_m(p)} -1}{p-1} .
\]  
\end{theorem}

By using this result, we determine the number of TAV groups of order square free. 

\begin{theorem}\label{numtavgroupsquarefree}
Let $n \in {\mathbb N}$ be square free. 
The number of TAV groups of order $n$ is
\[
\sum_{m|n} \, \prod_{p \in \pi(\frac{n}{m})} \frac{p^{c_m(p)} -1}{p-1} 
- \sum_{p \in \pi (n)} (2^{\bar{c}_n(p)} -1)
- 1.
\]  
\end{theorem}

\begin{proof}
We count finite groups $G$ of order $n$ whose commutator subgroups are trivial or cyclic groups of prime order, 
since $G$ is not a TAV group in these cases. 
If the commutator subgroup of $G$ is trivial, then $G$ is a cyclic group $C_n$. 
If the commutator subgroup of $G$ is a cyclic group of order prime $p$, then $G$ is expressed as 
$C_p \rtimes H$, where $H$ is a subgroup of ${\rm Aut} (C_p) \simeq C_{p-1}$. 
We denote $\bar{c}_n(p)$ by $k$, namely,  
primes $q_1, q_2, \ldots, q_k \in \pi(n)$ divide $p-1$.  
Then $H (< C_{p-1} \simeq {\mathbb Z}/(p-1) {\mathbb Z})$ is a cyclic group generated by 
\[
\frac{p-1}{{q_1}^{e_1} {q_2}^{e_2} \cdots {q_k}^{e_k}} ,
\]
where $e_1, e_2,\ldots , e_k = 0$ or $1$. 
When $e_1 = e_2 = \cdots = e_k = 0$, $H$ is trivial and the commutator subgroup is not $C_p$. 
Therefore there exist $2^k - 1$ nontrivial subgroups of ${\rm Aut} (C_p)$ and the commutator subgroups are $C_p$.
\end{proof}

\begin{example}
By Theorem \ref{holderthm}, there are $144$ finite groups of order $30030 = 2 \cdot 3 \cdot 5 \cdot 7 \cdot 11 \cdot 13$. 
In order to apply Theorem \ref{numtavgroupsquarefree}, we compute $\bar{c}_{30030}(p)$ for $p \in \pi (30030) = \{2,3,5,7,11,13\}$:  
\begin{align*}
&\bar{c}_{30030}(2)  = |\{\}|= 0, \, 
\bar{c}_{30030}(3) = |\{2\}| = 1, \, 
\bar{c}_{30030}(5) = |\{2\}| = 1, \\
&\bar{c}_{30030}(7) = |\{2,3\}|= 2, \, 
\bar{c}_{30030}(11) = |\{2,5\}| = 2, \, 
\bar{c}_{30030}(13) = |\{2,3\}| = 2. 
\end{align*}
Then the number of TAV groups of order $30030$ is 
\[
144 - (0 + 1 + 1 + 3 + 3 +3 ) - 1 = 132.
\] 

In fact, we can find the non-TAV groups of order $30030$ as follows.
A cyclic group $C_{30030}$ is not a TAV group. 
If a finite group can be expressed as $C_{13} \rtimes H$ 
whose commutator subgroup is $C_{13}$, 
$H$ is a subgroup of ${\rm Aut} (C_{13}) \simeq C_{12} \simeq {\mathbb Z}/12 {\mathbb Z}$ and 
is generated by $\frac{12}{2},\frac{12}{3},\frac{12}{2 \cdot 3}$. 
Then we obtain finite groups of order $30030$  
\[
(C_{13} \rtimes C_2) \times C_{1155}, 
(C_{13} \rtimes C_3) \times C_{770}, 
(C_{13} \rtimes C_6) \times C_{385} 
\]
respectively, which are not TAV groups, 
since their commutator subgroups are $C_{13}$. 
Similarly, we have non-TAV groups of order $30030$  
\begin{align*}
&(C_{11} \rtimes C_2) \times C_{1365}, 
(C_{11} \rtimes C_5) \times C_{546}, 
(C_{11} \rtimes C_{10}) \times C_{273}, \\
&(C_{7} \rtimes C_2) \times C_{2145}, 
(C_{7} \rtimes C_3) \times C_{1430}, 
(C_{7} \rtimes C_6) \times C_{715}, \\
&(C_{5} \rtimes C_2) \times C_{3003}, 
(C_{3} \rtimes C_2) \times C_{5005}.
\end{align*}
Therefore the remaining $132$ finite groups of order $30030$ are TAV groups. 
\end{example}

\section{Knots admitting a given TAV group}\label{sec:4}

In this section, we discuss knots whose twisted Alexander polynomials for a given TAV group is zero. 
Firstly, we explain how to construct knots admitting a given TAV group. 
Secondly, we show the existence of infinitely many such hyperbolic non-fibered knots for a given TAV group. 
Thirdly, we give a concrete example of a knot admitting a TAV group of order product of three distinct primes.

\subsection{Knots admitting a given TAV group}\label{subsec:3.0}
Let $G$ be a finite group with $w(G) = 1$ and assume that $G'$ is not a $p$-group. By Theorem \ref{thm:main-4} (\cite[Theorem~1.5]{IMS23-1}), $G$ is a TAV group of some knot. We can theoretically find such a knot by tracking the proof given in \cite{IMS23-1}, but it seems an impractically hard work and the resultant knot would be too complicated to examine. In this subsection, we give a more practical recipe to construct a knot admitting a given TAV group.

In order to explain the procedure, we recall the satellite construction of knots following \cite{CS16-1}. Let $K$ be a knot and $\alpha$ an unknotted loop disjoint from $K$. For another knot $J$, we glue the exteriors $E_\alpha$ of $\alpha$ and $E_J$ of $J$ by an orientation-reversing homeomorphism between the boundary tori so that a meridian and a preferred longitude of $\alpha$ are respectively identified with a preferred longitude and a meridian of $J$. Since the obtained $3$-manifold is homeomorphic to $S^3$, we can regard $K \subset E_\alpha$ as a new knot in $S^3$, which we denote by $K(\alpha, J)$. See Figure \ref{fig:satellite} in Subsection \ref{subsec:4.3} for an example, where $K= T(2, 15)$ and $J = T(3,5)$. 

Let $i \colon E_{K \cup \alpha} \to E_K$ and $j \colon E_{K \cup \alpha} \to E_{K(\alpha, J)}$ be the inclusion maps, which induce group homomorphisms $i_* \colon G(K \cup \alpha) \to G(K)$ and $j_* \colon G(K \cup \alpha) \to G(K(\alpha, J))$, where $E_{K\cup\alpha}=S^3\setminus\nu(K\cup\alpha)$ and $G(K\cup\alpha)=\pi_1(E_{K\cup\alpha})$. By \cite[Lemma 4.1]{CS16-1}, there exists a group homomorphism $\psi: G(K(\alpha, J)) \to G(K)$ such that $i_* = \psi \circ j_*$: $\psi$ is constructed by gluing the inclusion homomorphism $i_* \colon G(K \cup \alpha) \to G(K)$ and $\alpha_* \circ \phi_J \colon G(J) \to G(K)$, where $\phi_J \colon G(J) \to \mathbb{Z}$ is the abelianization map and $\alpha_* \colon \mathbb{Z} \to G(K)$ is the group homomorphism that sends $1 \in \mathbb{Z}$ to the element represented by the loop $\alpha$.

Let us see how to construct a knot admitting a given TAV group. Let $G$ be any TAV group. Since $G'$ has non-prime-power order, by \cite[Lemma~1]{gru}, $G'$ contains at least one of (i) a cyclic group of non-prime-power order or (ii) a non-abelian group of weight one; 
%$1$; 
take such a subgroup and denote it by $H$.

Since $w(G) = 1$, there exists a knot $K$ admitting an epimorphism $f_0\colon G(K) \to G$ onto $G$. Since $H$ is also of weight one, 
%$1$, 
we can take a loop $\alpha \subset E_K$ that bounds a disc in $S^3$ such that ${\rm lk}(K, \alpha) = 0$ and $f_0(\alpha)$ is contained in $H$ and normally generates $H$ (in $H$). In the case (i), we take a knot $J$ with $\Delta_J(e^{2 \pi i/ |H|}) = 0$ and define $f = f_0 \circ \psi \colon G(K(\alpha, J)) \to G$. In the case (ii), we take a knot $J$ that admits an epimorphism $f_J \colon G(J) \to H$ such that $f_J(m) = f_0(\alpha)$ and $f_J(\ell) = e$, where $(m, \ell)$ is a meridian-longitude pair of $J$, and then obtain a homomorphism $f \colon G(K(\alpha, J)) \to G$ by gluing $f_0$ and $f_J$.

\begin{theorem}\label{const-thm}
We have $\Delta_{K(\alpha, J)}^{\rho \circ f}(t) = 0$. In particular, $G$ is a TAV group of the knot $K(\alpha, J)$.
\end{theorem}

In the two propositions below, we examine when $\Delta_{K(\alpha, J)}^{\rho \circ f}(t) = 0$ and then Theorem \ref{const-thm} follows from them immediately: The case (i) follows from Proposition \ref{const-prop-1} and the case (ii) from Proposition \ref{const-prop-2}.

\begin{remark}
By \cite[Lemma~1]{gru}, we can take $H$ as follows: (i) a cyclic group of order $pq$ with distinct primes $p, q$ or (ii) a non-abelian metabelian group $C_p^n \rtimes C_q$ of weight one, 
%$1$, 
where $p, q$ are distinct primes and $n$ is a positive integer. If we choose $H$ in these ways, we can take the torus knot $T(p,q)$ as $J$.
\end{remark}

Let $G$ be a finite group and $f \colon G(K(\alpha, J)) \to G$ a homomorphism. Since $E_{K(\alpha, J)} = E_{K \cup \alpha} \cup E_J$, we have the restriction homomorphism $f_J\colon G(J) \to G$. The homomorphism $f$ factors through $\psi$, i.e., there exists a homomorphism $f_0 \colon G(K) \to G$ such that $f = f_0 \circ \psi$, if and only if the image of $f_J$ is cyclic. 

\begin{proposition}\label{const-prop-1}
Suppose that the linking number ${\rm lk}(K, \alpha)$ equals zero. We assume that $f$ factors through $\psi$, and let $d$ be the order of the cyclic group $f_J(G(J))$. Then, $\Delta^{\rho \circ f}_{K(\alpha, J)}(t) = 0$ if and only if $\Delta^{\rho \circ f_0}_K (t) = 0$ or $\Delta_J(e^{2 \pi i /d} ) = 0$, where $\Delta_J (t)$ is the Alexander polynomial of $J$.
\end{proposition}
\begin{proof}
By \cite[Theorem 3.1]{KSW05-1}, there exists a polynomial $h(t) \in \mathbb{Z}[t^{\pm 1}]$ such that $\Delta^{\rho \circ f}_{K(\alpha, J)}(t) = \Delta^{\rho \circ f_0}_K (t) \cdot h(t)$; since $\Delta^{\rho \circ f_0}_K (t) = 0$ implies that $\Delta^{\rho \circ f}_{K(\alpha, J)}(t) = 0$, we assume $\Delta^{\rho \circ f_0}_K (t) \neq 0$ and show that $\Delta^{\rho \circ f}_{K(\alpha, J)}(t) = 0$ if and only if $\Delta_J(e^{2 \pi i /d} ) = 0$.

Let $\tilde{E}$ be the covering space of $E_{K(\alpha, J)}$ corresponding to $f \times \phi \colon G(K(\alpha, J)) \to G \times \mathbb{Z}$, i.e., let $\tilde{E}_{K(\alpha, J)}$ be the universal covering space of $E_{K(\alpha,J)}$ 
%$E(K(\alpha, J))$ 
and define
$$\tilde{E} = (G \times \mathbb{Z}) \times_{f \times \phi} \tilde{E}_{K(\alpha, J)},$$
where $G$ and $\mathbb{Z}$ are equipped with the discrete topologies; remark that $\tilde{E}$ is not necessarily connected. It is sufficient to show that $H_1(\tilde{E})$ has infinite rank over $\mathbb{Z}$ (see Remark \ref{rmk:torsion}). Let $\pi_{K(\alpha, J)} \colon \tilde{E} \to E_{K(\alpha, J)}$ denote the covering map, and set
$$\tilde{E}_{K(\alpha, J)} = \pi_{K(\alpha, J)}^{-1}(E_{K \cup \alpha}), \qquad \tilde{E}_J = \pi_{K(\alpha, J)}^{-1}(E_J),$$
recalling that $E_{K(\alpha, J)} = E_{K \cup \alpha} \cup E_J$ as above.

Since $\tilde{E} = \tilde{E}_{K\cup \alpha} \cup \tilde{E}_J$, we have the Mayer-Vietoris exact sequence
$$\xymatrix@R=0cm{\cdots \ar[r] & H_1(\tilde{E}_{K \cup \alpha} \cap \tilde{E}_J) \ar[r]^-{\iota = (\iota_1, \iota_2)} & H_1(\tilde{E}_{K \cup \alpha}) \oplus H_1(\tilde{E}_J) \ar[r] & H_1(\tilde{E}) & \\
\ar[r]^-\partial & H_0(\tilde{E}_{K \cup \alpha} \cap \tilde{E}_J) \ar[r] & H_0(\tilde{E}_{K \cup \alpha}) \oplus H_0(\tilde{E}_J) \ar[r] & H_0(\tilde{E}) \ar[r] & 0.}$$
Because ${\rm lk}(K, \alpha) = 0$ and $f_J(G(J))$ is cyclic of order $d$, each connected component of $\tilde{E}_J$ is the $d$-fold cyclic covering space of $E_J$; in particular, the boundary is also connected and hence the inclusion homomorphism of $\tilde{E}_{K \cup \alpha} \cap \tilde{E}_J \subset \tilde{E}_J$ at dimension $0$ is injective. This implies that the connecting homomorphism $\partial\colon H_1(\tilde{E}) \to H_0(\tilde{E}_{K \cup \alpha} \cap \tilde{E}_J)$ equals zero. Thus, $H_1(\tilde{E})$ is isomorphic to the cokernel of the inclusion homomorphism $\iota \colon H_1(\tilde{E}_{K \cup \alpha} \cap \tilde{E}_J) \to H_1(\tilde{E}_{K \cup \alpha}) \oplus H_1(\tilde{E}_J)$.

Let $m \subset \partial E_J$ be a meridional loop and $\ell \subset \partial E_J$ a longitudinal loop with ${\rm lk}(J, \ell) = 0$. Let $M, L \subset H_1(\tilde{E}_{K\cup\alpha} \cap \tilde{E}_J)$ be the subspaces generated by the homology classes represented by the lifts of $m^d,\ell$, respectively. Since $\tilde{E}_{K\cup\alpha} \cap \tilde{E}_J = \partial E_J$, $H_1(\tilde{E}_{K\cup\alpha} \cap \tilde{E}_J) = M \oplus L$. Because each connected component of $\tilde{E}_J$ is the $d$-fold cyclic covering space of $E_J$, $\iota_2|_L \colon L \to H_1(\tilde{E}_J)$ is zero and $\iota_2|_M \colon M \to H_1(\tilde{E}_J)$ is injective. Hence ${\rm Coker} \; \iota$ contains a subspace isomorphic to $H_1(\tilde{E}_{K \cup \alpha})/\iota_1(L) \cong H_1(\tilde{E}_K)$, where $\tilde{E}_K$ is the covering space of $E_K$ corresponding to $f_0 \times \phi\colon G(K) \to G \times \mathbb{Z}$, and the quotient of ${\rm Coker}\;\iota$ over the subspace is isomorphic to the homology group $H_1(\widehat{E}_J)$ of $\widehat{E}_J$, the $3$-manifold obtained by attaching solid tori to the boundary components of $\tilde{E}_J$ so that the lifts of $m^d$ bound discs in $\widehat{E}_J$.

As we assume $\Delta^{\rho \circ f_0}_K (t) \neq 0$, the subspace isomorphic to $H_1(\tilde{E}_K)$ is of finite rank over $\mathbb{Z}$. Furthermore, because each connected component of $\widehat{E}_J$ is homeomorphic to the $d$-fold branched cyclic cover $B_J^d$ of $E_J$, $H_1(\widehat{E}_J)$ is isomorphic to the direct sum of infinitely many copies of $H_1(B_J^d)$. Therefore we have $\Delta^{\rho \circ f}_{K(\alpha, J)}(t) = 0$, i.e., ${\rm rank}_{\mathbb{Z}} \, H_1(\tilde{E}) = \infty$, if and only if $H_1(B_J^d)$ has positive rank, which is equivalent to the condition $\Delta_J(e^{2 \pi i/d}) = 0.$
\end{proof}

\begin{proposition}\label{const-prop-2}
Suppose that the linking number ${\rm lk}(K, \alpha)$ equals zero. If $f \colon G(K(\alpha, J)) \to G$ does not factor through $\psi$, $\Delta^{\rho \circ f}_{K(\alpha, J)}(t) = 0$.
\end{proposition}
\begin{proof}
By the Seifert-van Kampen theorem, the knot group $G(K(\alpha, J))$ is the amalgamated free product $G(K \cup \alpha) *_{\pi_1(\partial E_J)} G(J)$. Let $m, \ell \in G(J)$ denote the meridian and the preferred longitude of $J$, and let $d_1, d_2$ be the orders of $f(m)$ and $f(\ell)$, respectively. We set
$$\xi = \left(\sum_{j=0}^{d_1-1}(f(m)^j,0)\right) \left(\sum_{j=0}^{d_2-1} (f(\ell)^j,0)\right) \in \mathbb{Z}[G \times \mathbb{Z}],$$
and define homomorphisms $\tilde{f}_{K \cup \alpha} \colon G(K \cup \alpha) \to \mathbb{Z}[G \times \mathbb{Z}] \rtimes (G \times \mathbb{Z})$ and $\tilde{f}_J \colon G(J) \to \mathbb{Z}[G \times \mathbb{Z}] \rtimes (G \times \mathbb{Z})$ by
\begin{align*}
\tilde{f}_{K \cup \alpha}(u) &= (0; f(u), \phi(u)) & (u \in G(K \cup \alpha)),\\
\tilde{f}_J(u) &= \bigl(((f(u),0) - (e,0))\xi; f(u), 0\bigr) & (u \in G(J)).
\end{align*}
Because $m$ commutes with $\ell$, $((f(u),0) - (e,0))\xi = 0$ for $u \in \pi_1(\partial E_J)$. Therefore we obtain a homomorphism $\tilde{f} \colon G(K(\alpha, J)) \to \mathbb{Z}[G \times \mathbb{Z}] \rtimes (G \times \mathbb{Z})$ by gluing $\tilde{f}_{K \cup \alpha}$ and $\tilde{f}_J$.

By the definitions of $\tilde{f}_{K \cup \alpha}$ and $\tilde{f}_J$, $\tilde{f}$ is a lift of $f \times \phi$ (remark that $\phi$ is zero on $G(J)$ since ${\rm lk}(K, \alpha) = 0$). Furthermore we should remark that the image $f(E_J) \subset G$ is not generated by the two elements $f(m), f(\ell)$; if generated, $f(E_J)$ is an abelian group of weight one and hence cyclic, which contradicts the assumption that $f$ does not factor through $\psi$. Therefore we can take $u_1 \in G(J)$ such that $f(u_1) \not\in \langle f(m), f(\ell) \rangle$. Let $u_2 \in G(K \cup \alpha)$ be a meridional loop of $K$ and $d_3, d_4$ the orders of $f(u_1 u_2), f(u_2) \in G$. Then, we have
\begin{eqnarray*}
\tilde{f}((u_1u_2)^{d_3 d_4} u_2^{-d_3 d_4}) &=& \bigl(((f(u_1),0) - (e, 0))\xi; f(u_1 u_2), 1\bigr)^{d_3 d_4} (0; e, -d_3 d_4)\\
&=& \left( \sum_{j=0}^{d_3 d_4 -1} (f(u_1u_2)^j, j) ((f(u_1),0) - (e, 0))\xi; e, 0 \right).
\end{eqnarray*} 
Here, we note that the product in $\mathbb{Z}[G \times \mathbb{Z}] \rtimes (G \times \mathbb{Z})$ is given by $(\lambda;x,a)(\lambda';x',a')=(\lambda+(x,a)\lambda';xx',a+a')$, where $\lambda,\lambda'\in \Z[G\times \Z], x,x'\in G$, and $a,a'\in\Z$. 
Since $f(u_1) \not\in \langle f(m), f(\ell) \rangle$, $\sum_{j=0}^{d_3 d_4 -1} (f(u_1u_2)^j, j) ((f(u_1),0) - (e, 0))\xi \neq 0$. Therefore $\tilde{f}$ is a nontrivial lift and we have $\Delta^{\rho \circ f}_{K(\alpha, J)}(t) = 0$ by Theorem \ref{lifting-thm}.
\end{proof}

\begin{remark}
If ${\rm lk}(K, \alpha) \neq 0$, then $\Delta_{K(\alpha, J)}^{\rho \circ f}(t) = 0$ if and only if $\Delta_{K\cup \alpha}^{\rho \circ f}(t) = 0$ or $\Delta_J^{\rho \circ f}(t) = 0$. %Especially 
In particular, if $f$ maps a preferred longitude of $J$ to the identity, e.g., if $f$ factors through $\psi$, then $f$ induces a homomorphism $f_0 \colon G(K) \to G$ and the condition $\Delta_{K\cup \alpha}^{\rho \circ f}(t) = 0$ is equivalent to $\Delta_K^{\rho \circ f_0}(t) = 0$. A proof of these facts is similar to that of Proposition \ref{const-prop-1}; a different point is that $H_1(\tilde{E}_{K \cup \alpha} \cap \tilde{E}_J)$ is of finite rank if ${\rm lk}(K, \alpha) \neq 0$.
\end{remark}

\subsection{Existence of hyperbolic knots for a given TAV group}\label{sec:4.1}

In the previous subsection, we constructed a knot admitting a given TAV group. 
The purpose of this subsection is to prove the following theorem: 

\begin{theorem}[Theorem \ref{thm:main-5}]\label{thm:hyperbolic}
For any TAV group $G$, there are infinitely many hyperbolic knots that admit $G$ as a TAV group. Moreover, if $G$ is realized as a smallest TAV group of a knot, then infinitely many hyperbolic knots also admit $G$ as a smallest TAV group. 
%Moreover, if $G$ is realized as a smallest TAV group, we can take them so that they have $G$ as a smallest TAV group.
\end{theorem}

\begin{proof}
Let $K$ be a knot admitting $G$ as a TAV group and take an $n$-braid $b$ that presents $K$. Here, we may assume that $b$ is a pseudo-Anosov element in the braid group $B_n$. For example, if $n$ is prime, this assumption is satisfied; the primeness implies the irreducibility, and $b$ is not periodic since the closure $K$ is not a torus link. For a positive integer $k$ coprime to $n$, let $K_k$ be the knot obtained by taking the closure of the braid $b^k$. We first show that $G$ is a TAV group of $K_k$, and under the assumption $\ord(K) = |G|$ we prove that $\ord(K_k) = \ord(K)$ holds if $k$ is coprime to the natural numbers less than or equal to $n\cdot\ord(K)^n$. Finally, we see that $K_k$ is hyperbolic if $k$ is sufficiently large.

Let $p_k \colon G(K_k) \to G(K)$ denote the quotient map. For an epimorphism $f \colon G(K) \to G$ with $\Delta_K^{\rho \circ f}(t) = 0$, $\Delta_{K_k}^{\rho \circ (f \circ p_k)}(t) = 0$ since $\Delta_K^{\rho \circ f}(t)$ is a factor of $\Delta_{K_k}^{\rho \circ (f \circ p_k)}(t)$ by \cite[Theorem~1]{HLN06-1}; $G$ is a TAV group of $K_k$.

Let us consider the case $\ord(K) = |G|$. As shown above, we have $\ord(K_k)\leq\ord(K)$ for any $k$.
Let us assume that $k$ is coprime to the natural numbers less than or equal to $n\cdot\ord(K)^n$ and show the other inequality.

We first claim that any homomorphism from $G(K_k)$ to a group $H$ with $|H| < \ord(K)$ factors through the quotient map $p_k \colon G(K_k) \to G(K)$. To see this, we consider the Hurwitz action of $B_n$ on $H^n$. We should recall that there is a one-to-one correspondence between the fixed point set of $b \;\text{(resp. $b^k$)} \in B_n$ and the set of the group homomorphisms from $G(K)$ (resp. $G(K_k)$) to $H$. By the assumption, $k$ is coprime to $|H^n|!$ and hence is coprime to the order of $b$. Therefore the fixed point set of $b$ is equal to that of $b^k$; in other words, any homomorphism from $G(K_k)$ to $H$ factors through $G(K)$.

Thus, any homomorphism $f_k: G(K_k) \to H$ is the composite of the quotient map $p_k$ and a homomorphism $f: G(K) \to H$ if $|H| < \ord(K)$. By \cite[Theorem~3]{HLN06-1}, there exists a polynomial $F(t, s) \in \mathbb{C}[t^{\pm 1}, s^{\pm 1}]$ such that
$$\Delta_{K_k}^{\rho \circ f_k}(t) = \Delta_K^{\rho \circ f}(t) \prod_{j = 1}^{k-1} F(t, \zeta_k^j),$$
where $\zeta_k \in \mathbb{C}$ is a primitive $k$-th root of unity. As in the proof of \cite[Lemma~19]{HLN06-1}, $\Delta_K^{\rho \circ f}(t) \neq 0$ implies that $F(t, s) \neq 0$. Moreover, the width of $F(t, s)$ with respect to $s$ is less than or equal to $n|H|$ since $F(t, s)$ is the two-variable twisted Alexander polynomial $\Delta_{K \cup \alpha}^{\rho \circ f}(t, s)$ in the sense of \cite{Wada94-1}, where $\alpha$ is a trivial knot component around $b$. Since $k$ is coprime to any natural number less than or equal to $n|H|$, we have $F(t, \zeta_k^j) \neq 0$ for $j = 1, \dots, k-1$ and hence $\Delta_{K_k}^{\rho \circ f_k}(t) \neq 0$, which concludes that $H$ is not a TAV group of $K_k$ if $|H| < \ord(K)$, i.e., $\ord(K_k) \geq \ord(K)$.

Let us show that the knot $K_k$ is hyperbolic for sufficiently large $k$ using Thurston's hyperbolic Dehn surgery theorem. We regard $b$ as a self-homeomorphism of the $(n+1)$-punctured sphere, and construct a mapping torus $M$ with monodromy $b$. Then, $M$ becomes a hyperbolic $3$-manifold, 
because $b$ is a pseudo-Anosov element in $B_n$. If we take a trivial loop $\alpha$ around $b$, then $K\cup\alpha$ becomes a hyperbolic link. Hence, for a sufficiently large $k\in\N$, the hyperbolic Dehn surgery theorem implies that $M_{(\infty,k/0)}$ admits the hyperbolic structure. Here, $(\infty,k/0)$ denotes the attaching slopes for $K$ and $\alpha$ respectively. For example, $M_{(\infty,1/0)}$ denotes $S^3\setminus K$. 

When $k>1$ the loop $\alpha$ becomes a cone singularity, but if we take the $k$-fold cyclic branched covering along $\alpha$, we can obtain the usual hyperbolic structure.  Then, the resulting manifold coincides with $S^3\setminus K_k$, and hence, $K_k$ is a hyperbolic knot as desired. 
\end{proof}

Using \cite[Theorem 3.2]{MS22-1}, \cite[Theorem 1.3]{IMS23-1}, and Theorem \ref{thm:hyperbolic}, we have the following corollary (see \cite[Corollary 1.4]{IMS23-1} for the case $\ord(K)=24$). 

\begin{corollary}\label{cor:hyperbolic}
There are infinitely many hyperbolic knots $K\in\mathcal{N}$ with $\ord(K)=24,60,96$, or $120$.
\end{corollary}

\subsection{Infinitely many hyperbolic knots $K$ with $\ord(K) =pqr$}\label{subsec:4.3}

\begin{theorem}[Theorem \ref{thm:main-6}]\label{thm:pqr}
Let $G$ be a TAV group of order $pqr$, where $p,q,r$ are primes.
Then there exist infinitely many hyperbolic knots that admit $G$ as a smallest TAV group;
in particular, there exist infinitely many hyperbolic knots $K$ with $\ord(K) = pqr$.
\end{theorem}

\begin{figure}[t]
\centering
\includegraphics[width=8cm]{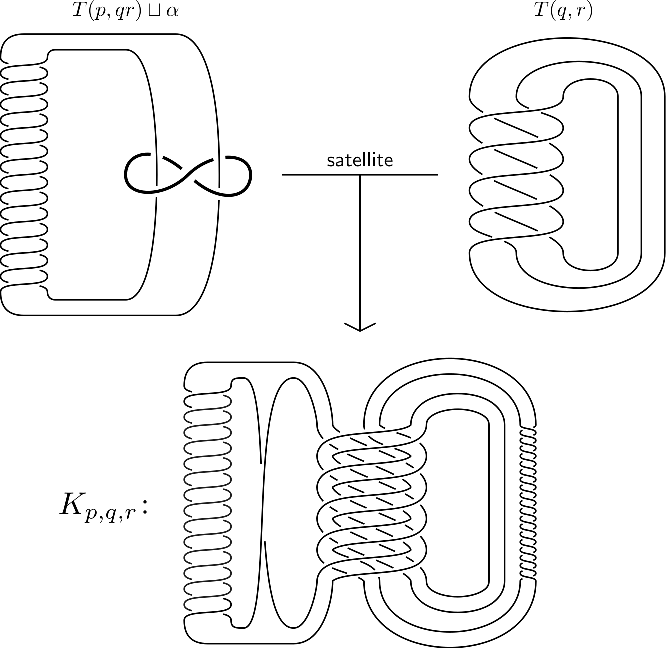}
\caption{Construction of $K_{p,q,r}$; $(p,q,r) = (2,3,5)$.}\label{fig:satellite}
\end{figure}

For a TAV group $G$ of order $pqr$, we construct a satellite knot as in Subsection \ref{subsec:3.0}: Let us take the torus knots $T(p, qr)$ and $T(q, r)$ as $K$ and $J$, respectively. Let $f_0 \colon G(K) \to G$ be a surjective homomorphism. We take an unknotted loop $\alpha$ disjoint from $K$ such that the linking number ${\rm lk}(K, \alpha)$ is zero and $f_0$ sends $\alpha$ to a generator of the commutator subgroup $C_{qr}$ of $G$. We denote the resulting satellite knot $K(\alpha, J)$ by $K_{p,q,r}$; see Figure \ref{fig:satellite} for an example. Let $f \colon G(K_{p,q,r}) \to G$ be the composite $G(K_{p,q,r}) \xrightarrow{\psi} G(K) \xrightarrow{f_0} G$.

\begin{proposition}\label{prop:pqr}
We have $\Delta_{K_{p,q,r}}^{\rho \circ f}(t) = 0;$ in particular, $G$ is a TAV group of $K_{p,q,r}$.
\end{proposition}
\begin{proof}
Since $G'$ is a cyclic group of order $qr$, we can apply Proposition \ref{const-prop-1} to $\Delta_{K_{p,q,r}}^{\rho \circ f}(t)$: Recalling that the Alexander polynomial $\Delta_J(t)$ of $J = T(q,r)$ is the $qr$-th cyclotomic polynomial, we have $\Delta_J(e^{2 \pi i/qr}) = 0$ and hence $\Delta_{K_{p,q,r}}^{\rho \circ f}(t) = 0.$
\end{proof}

In order to prove Theorem \ref{thm:pqr}, we prepare a purely algebraic lemma:

\begin{lemma}\label{lem:qr-grp}
Let $q, r$ be distinct primes with $q < r$ and let $H$ be a finite group that is generated by $x \in H$ with order $q$ and $y \in H$ with order $r$. If $|H| \leq q^2r/2$, $H$ is a metacyclic group of order $qr$.
\end{lemma}
\begin{proof}
Because $H$ has elements with orders $q$ and $r$, the order of $H$ is a multiple of $qr$; define $n = |H|/qr$. Let $Q, R$ be the subgroup generated by $x, y$, respectively. Since $n < q < r$, $Q$ and $R$ are Sylow subgroups.

We claim that at least one of $Q$ or $R$ is normal in $H$. Let $n_q, n_r$ be the numbers of the Sylow $q$-subgroup and Sylow $r$-subgroups of $H$, respectively, and suppose that $n_q, n_r >1$. By the Sylow theorem, $n_q \in q\mathbb{Z} + 1$ and especially $n_q > q$. Considering the transitive action of $H$ on the set of the Sylow $q$-subgroups, we find that $n_q$ divides $nr$; this implies that $r$ divides $n_q$ since $n_q > q > n$. Then, $n'_q := n_q/r$ is a positive integer less than or equal to $n$. In the same way, $n'_r := n_r/q$ is an integer with $1 \leq n'_r \leq n$.

Because $q$ and $r$ are coprime, we can take integers $a, b$ such that $ar + bq = qr + 1$, and we may further assume that $0 < a < q$ and hence $0 < b < r$. Here, the integer $a$ can be characterized as the unique integer such that $ar \in q \mathbb{Z} + 1$ and $0 < a < q$. Thus, we find $a = n'_q$ and similarly $b = n'_r$. We have $n_q + n_r = qr + 1$ and hence one of $n_q$ and $n_r$ is greater than $qr/2$. In either case, this means $|H| > qr/2 \cdot q = q^2r/2$, but this is a contradiction. Therefore $n_q = 1$ or $n_r = 1$.

Since $Q$ or $R$ is normal in $H$, we find $H = QR$. Recalling that $Q$ and $R$ are cyclic, we conclude that $H$ is a metacyclic group of order $qr$.
\end{proof}

\begin{proof}[Proof of Theorem \ref{thm:pqr}]
The existence of infinitely many hyperbolic knots admitting $G$ as a TAV group is a direct consequence of Proposition \ref{prop:pqr} and Theorem \ref{thm:hyperbolic}. After verifying $\ord(K_{p,q,r}) = pqr$, Theorem \ref{thm:hyperbolic} also shows that the family of hyperbolic knots contains infinitely many knots $K'$ with $\ord(K') = pqr$.

We first show that the image $H_J$ of any homomorphism $g_J \colon G(J) \to H$ from $G(J)$ to any group $H$ with $|H| < pqr$ is metabelian. Let us recall that $G(J) = G(T(q, r))$ has a well-known presentation $\langle x, y \mid x^q = y^r \rangle$ and that its center is the infinite cyclic group generated by $x^q = y^r$. Thus, $H_J/Z(H_J)$, where $Z(H_J)$ is the center of $H_J$, is generated by two elements $\overline{g_J(x)}$ and $\overline{g_J(y)}$, where an over-line expresses the image in the quotient group $H_J/Z(H_J)$, and we find $\overline{g_J(x)}^q = \overline{g_J(y)}^r = \overline{e}$. If $\overline{g_J(x)}$ or $\overline{g_J(y)}$ is the identity, $H_J$ is generated by a single element and hence cyclic. Otherwise, we can apply Lemma \ref{lem:qr-grp} (we should recall that $p \leq (q-1)/2$) to $H_J/Z(H_J)$ to find that $H_J/Z(H_J)$ is a metacyclic group of order $qr$. Because the second homology group of a metacyclic group of order $qr$ is trivial, any central extension preserves the commutator subgroup and only extends the abelianization. Thus, $H_J$ is metabelian. Furthermore, this proof shows that if $H_J$ is not cyclic, then $|H|$ is a multiple of $qr$.

Let $g\colon G(K_{p, q, r}) \to H$ be an epimorphism onto a group $H$ with $|H| < pqr$. Since $E_{K_{p,q,r}} = E_{K \cup \alpha} \cup E_J$, we have the restriction homomorphism $g_J \colon G(J) \to H$ and its image $H_J$ is metabelian as seen above; in particular, the second derived group $(H_J)''$ is trivial. Recalling that a preferred longitude of $J$ and a meridian of $\alpha$ are glued together in $E_{K_{p,q,r}}$, we find that $g$ induces a group homomorphism $g_K \colon G(K) \to H$.

If $|H| \in qr \mathbb{Z}$, $p$ does not divide $|H|$ and hence the image of $g_K\colon G(K) = G(T(p, qr)) \to H$ is cyclic. In this case, (a preferred longitude of) $\alpha$ is mapped to $e \in H$. Since a preferred longitude of $\alpha$ is identified with a meridian of $J$, $g_J$ is trivial and it follows that $H$ is cyclic; especially, $H$ is not a TAV group.

If $|H| \not\in qr \mathbb{Z}$, $H_J$ is cyclic as seen above, which also implies that $g \colon G(K_{p,q,r}) \to H$ factors through $\psi \colon G(K_{p,q,r}) \to G(K)$, i.e., there exists a homomorphism $g_0\colon G(K) \to H$ such that $g = g_0 \circ \psi$. Since $K = T(p,qr)$ is fibered, $\Delta_K^{\rho \circ g_0}(t) \neq 0$. Also, because the roots of the Alexander polynomial $\Delta_J(t)$ of $J = T(q,r)$ are the primitive $qr$-th roots of unity, the assumption $|H| \not\in qr\mathbb{Z}$ implies that $\Delta_J(e^{2 \pi i/|H_J|}) \neq 0$. Thus, Proposition \ref{const-prop-1} concludes that $\Delta_{K_{p,q,r}}^{\rho \circ g}(t) \neq 0$: $H$ is not a TAV group of $K_{p,q,r}$.
\end{proof}

\begin{corollary}\label{cor:273}
There exist infinitely many hyperbolic knots that admit distinct smallest TAV groups. Moreover, there exist infinitely many TAV orders $\ord(K)$ which correspond to such smallest TAV groups. 
\end{corollary}

\begin{proof}
Let $G_i~(i=1,2)$ be the finite groups defined by
\begin{align*}
G_1&=C_{91}\rtimes C_3=\la x,y\,|\,x^{91}, y^3,yxy^{-1}x^{-74}\ra, \\
G_2&=C_{91}\rtimes_4 C_3=\la x,y\,|\, x^{91},y^3,yxy^{-1}x^{-9}\ra.
\end{align*} 
They appear in GroupNames \cite{GN} as $(273,3)$ and $(273,4)$
respectively, and share the order $273=3\cdot7\cdot13$. 

If we apply Theorem \ref{thm:pqr} to these TAV groups $G_1$ and $G_2$, we can obtain infinitely many hyperbolic knots as desired. Because the proof of Theorem \ref{thm:pqr} depends only on the order of these TAV groups. 
Since it is well known that there exist infinitely many prime numbers $q$ satisfying $q\equiv1~(\mathrm{mod} ~3)$, the latter assertion follows immediately.
\end{proof}

In our previous paper \cite{IMS23-1}, we proposed the following problem:

\begin{problem}[{\cite[Problem~5.3]{IMS23-1}}]\label{que:5.3}
Find a non-fibered knot $K$ and an epimorphism $f \colon G(K) \to D_{15}$ such that $\D_K^{\rho\circ f}(t)=0$. 
\end{problem}

As a corollary of Theorem \ref{thm:pqr}, we can give an answer to Problem \ref{que:5.3}. 

\begin{corollary}\label{cor:D_15}
Let $K=K_{2,3,5}$ be the knot introduced after Theorem \ref{thm:pqr}. 
%with respect to the companion $T(3,5)$ and the pattern $T(2,15)$. 
Then, the following hold:
\begin{itemize}
\item[(i)]
There is an epimorphism $f\colon G(K)\to D_{15}$ such that 
$\D_K^{\rho\circ f}(t)=0$, 
\item[(ii)]
For any epimorphism $f\colon G(K) \to G$ with $|G|<30$, 
$\D_K^{\rho\circ f}(t)\not=0$. 
\end{itemize}
In particular, it holds that $\ord(K)=30$. 
\end{corollary}

As a concluding remark of this section, we propose the following problem, 
which is a generalization of Theorem \ref{thm:pqr}. 

\begin{problem}
Let $n \in {\mathbb N}$ be square free. 
Suppose that there exists a TAV group of order $n$. 
Do there exist infinitely many  hyperbolic knots $K$ with $\ord (K) = n$? 
\end{problem}

\section{Faithful irreducible representation}\label{sec:5}

In this section, we show that for any faithful irreducible representation $\rho$ of a TAV group $G$ 
there exist a knot $K$ and an epimorphism $f \colon G(K) \to G$ with $\Delta_K^{\rho \circ f}(t) = 0$. 

It is known that evey representation of a finite group $G$ can be decomposed uniquely into a direct sum of irreducible representations of $G$. 
Especially, the regular representation of $G$ contains all the irreducible representations of $G$ as follows. 
Let $\rho$ be the regular representation of $G$ and $\rho_1,\rho_2, \ldots, \rho_k$ all the irreducible representations of $G$. 
It is known that $\rho$ is conjugate with a direct sum of $\rho_1,\rho_2, \ldots, \rho_k$ with multiplicities $\dim \rho_i$ 
and then $\Delta_K^{\rho \circ f}(t)$ can be expressed as  
the product of $\Delta_K^{\rho_1 \circ f}(t), \Delta_K^{\rho_2 \circ f}(t), \ldots , \Delta_K^{\rho_k \circ f}(t)$.  
Therefore $\Delta_K^{\rho \circ f}(t) = 0$ implies $\Delta_K^{\rho_i \circ f}(t) = 0$ for at least one irreducible representation $\rho_i$. 
In this sense, we should discuss which irreducible representation causes the twisted Alexander polynomial to be zero. 
As an answer for this question, we obtain the following. 

\begin{theorem}[Theorem \ref{thm:main-7}]\label{thm_faithfulirrrep}
Let $G$ be a TAV group and $\rho \colon G \to {\rm GL}(V)$ a faithful irreducible representation of $G$. Then, there exist a knot $K$ and an epimorphism $f \colon G(K) \to G$ such that $\Delta_K^{\rho \circ f}(t) = 0$.
\end{theorem}

In order to show Theorem \ref{thm_faithfulirrrep}, we begin by proving the following theorem, 
which provides a necessary and sufficient condition for the twisted Alexander polynomial to vanish. 

\begin{theorem}\label{thm_52}
Let $K$ be a knot and $m \in G(K)$ a meridional loop. Let $\rho \colon G(K) \to {\rm GL}(V)$ be a representation. The twisted Alexander polynomial $\Delta_K^\rho(t)$ is zero if and only if there exists a nontrivial lift $\tilde{\rho} \colon G(K) \to (V \otimes \mathbb{Z}[t^{\pm 1}]) \rtimes ({\rm GL}(V) \times \mathbb{Z})$ of $\rho \times \phi \colon G(K) \to {\rm GL}(V) \times \mathbb{Z}$ such that $\tilde{\rho}(m) = (0; \rho(m), 1)$.
\end{theorem}
\begin{proof}
We take a presentation $\langle x_1,\dots, x_s, x_{s+1} \mid r_1, \dots, r_s\rangle$ of $G(K)$ where $x_{s+1}$ corresponds to $m$. Using the definition in \cite{Wada94-1}, $\Delta_K^\rho(t) = 0$ if and only if the matrix $\Phi\left(\frac{\partial r_i}{\partial x_j}\right)_{i,j=1}^s$, regarded as a linear map from $(R[t^{\pm 1}])^{ns}$ to $(R[t^{\pm 1}])^{ns}$, has a nontrivial kernel.

By \cite[Proposition 1]{Wada94-1}, there exists a bijection between the kernel of the matrix $\Phi\left(\frac{\partial r_i}{\partial x_j}\right)^{1 \leq i \leq s+1}_{1 \leq j \leq s}$ regarded as a linear map from $(V \otimes \mathbb{Z}[t^{\pm 1}])^{s+1}$ to $(V \otimes \mathbb{Z}[t^{\pm 1}])^s$ and the set of the derivations of $G(K)$ with values in $V \otimes \mathbb{Z}[t^{\pm 1}]$. In this bijection, a vector in the kernel with last $n$ entries $0$ corresponds to a derivation which sends $m$ to $0$. Therefore we have $\Delta_K^\rho(t) = 0$ if and only if there exists a nontrivial derivation $f \colon G \to V \otimes \mathbb{Z}[t^{\pm 1}]$ with $f(m) = 0$. Since a set-theoretic lift $\tilde{\rho} = (f', \rho \times \phi) \colon G(K) \to (V \otimes \mathbb{Z}[t^{\pm 1}]) \rtimes ({\rm GL}(V) \times \mathbb{Z})$ of $\rho \times \phi \colon G(K) \to {\rm GL}(V) \times \mathbb{Z}$ is a group homomorphism if and only if $f' \colon G(K) \to V \otimes \mathbb{Z}[t^{\pm 1}]$ is a derivation, we obtain the theorem.
\end{proof}

By using Theorem \ref{thm_52}, 
we next show Lemma \ref{v1-lem} and Lemma \ref{v2-lem} that give sufficient conditions for the twisted Alexander polynomial to vanish.

\begin{lemma}\label{v1-lem}
Let $G$ be a finite group of weight one and $\rho \colon G \to {\rm GL}(V)$ a representation. If there exists an element $g \in G'$ of order $pq$, where $p$ and $q$ are distinct primes, such that $\rho(g) \in {\rm GL}(V)$ has a primitive $pq$-th root of unity as an eigenvalue, then there exist a knot $K$ and an epimorphism $f \colon G(K) \to G$ such that $\Delta_K^{\rho \circ f}(t) = 0$.
\end{lemma}
\begin{proof}
Since $G$ is of weight one, there exists a pair $(K_0, f_0)$ of a knot $K_0$ and an epimorphism $f_0 \colon G(K_0) \to G$. Let $J$ denote the $(p, q)$-torus knot $T(p,q)$ and define $f_J \colon G(J) \to \langle g \rangle \subset G$ by sending every meridional loop to $g$. Let $\alpha \subset E_{K_0}$ be a loop bounding a disk in $S^3 \supset E_{K_0}$ such that ${\rm lk}(K_0, \alpha) = 0$ and $f_0(\alpha) = g$. Let $K$ be the satellite knot $K_0(\alpha, J)$. By the definitions of $f_0$ and $f_J$, we can glue them to obtain $f \colon G(K) \to G$.

Because the Alexander polynomial of $J = T(p,q)$ is the $pq$-th cyclotomic polynomial $\varphi_{pq}(t)$, there exists a nontrivial lift $\tilde{\phi} = (\tilde{\phi}_1, \phi) \colon G(J) \to (\mathbb{Z}[t]/ (\varphi_{pq}(t))) \rtimes \mathbb{Z}$ of $\phi \colon G(J) \to \mathbb{Z}$ such that $\tilde{\phi}(m) = (0,1)$, where $m \in G(J)$ is a meridional loop of $J$. Then, define homomorphisms $\tilde{f}_0 \colon G(K_0 \cup \alpha) \to (V \otimes \mathbb{Z}[t^{\pm 1}]) \rtimes (G \times \mathbb{Z})$ and $\tilde{f}_J \colon G(J) \to (V \otimes \mathbb{Z}) \rtimes (G \times \mathbb{Z})$ by
\begin{align*}
\tilde{f}_0(u) &= (0; f_0(u), \phi(u)) & (u \in G(K_0 \cup \alpha)),\\
\tilde{f}_J(u) &= \bigl(\tilde{\phi}_1(u)(g_p g_q) \otimes 1; f_J(u), 0\bigr) & (u \in G(J)).
\end{align*}
Since $\tilde{\phi}_1(m)(t) = 0$, we obtain a homomorphism $\tilde{f} \colon G(K(\alpha, J)) \to (V \otimes \mathbb{Z}[t^{\pm 1}]) \rtimes (G \times \mathbb{Z})$ by gluing $\tilde{f}_0$ and $\tilde{f}_J$. By the definition of $\tilde{f}_0$, the first component of $\tilde{f}(m_K)$ for some meridional loop $m_K$ of $K$ ($K_0$) is zero, and $\tilde{f}$ is a nontrivial lift since $\tilde{\phi}$ is nontrivial. Therefore $\Delta_K^f(t) = 0$ by Theorem \ref{thm_52}.
\end{proof}

\begin{lemma}\label{v2-lem}
Let $G$ be a finite group of weight one and $\rho \colon G \to {\rm GL}(V)$ a faithful representation. If a non-commutative subgroup $H$ of $G'$ contains a weight element $h \in H$ such that $\rho(h) \in {\rm GL}(V)$ has $1$ as an eigenvalue, then there exist a knot $K$ and an epimorphism $f \colon G(K) \to G$ such that $\Delta_K^{\rho \circ f}(t) = 0$.
\end{lemma}
\begin{proof}
Since $G$ is of weight one, there exists a pair $(K_0, f_0)$ of a knot $K_0$ and an epimorphism $f_0 \colon G(K_0) \to G$. Also, we can take a knot $J$ and an epimorphism $f_J \colon G(J) \to H$ with $f_J(m) = h$ and $f_J(\ell) = e$, where $m, \ell \in G(J)$ are a meridian and a preferred longitude of $J$, respectively. Let $\alpha \subset E_{K_0}$ be a loop bounding a disk in $S^3 \supset E_{K_0}$ such that ${\rm lk}(K_0, \alpha) = 0$ and $f_0(\alpha) = h$. Let $K$ be the satellite knot $K_0(\alpha, J)$. By the definitions of $f_0$ and $f_J$, we can glue them to obtain $f \colon G(K) \to G$.

We take an element $h' \neq h$ of $H$ conjugate to $h$. Since $\rho$ is faithful, we can take a nonzero element $v_0 \in V$ such that $h v_0 = v_0$ and $h' v_0 \neq v_0$, and set $v = v_0 \otimes 1 \in V \otimes \mathbb{Z}[t^{\pm 1}]$. Then, define homomorphisms $\tilde{f}_0 \colon G(K_0 \cup \alpha) \to (V \otimes \mathbb{Z}[t^{\pm 1}]) \rtimes (G \times \mathbb{Z})$ and $\tilde{f}_J \colon G(J) \to (V \otimes \mathbb{Z}) \rtimes (G \times \mathbb{Z})$ by
\begin{align*}
\tilde{f}_0(u) &= (0; f_0(u), \phi(u)) & (u \in G(K_0 \cup \alpha)),\\
\tilde{f}_J(u) &= \bigl(((f_J(u),0) - (e,0))v; f_J(u), 0\bigr) & (u \in G(J)).
\end{align*}
Remark that $((f(u),0) - (e,0))v = 0$ for $u \in \pi_1(\partial E_J)$. Therefore we obtain a homomorphism $\tilde{f} \colon G(K(\alpha, J)) \to (V \otimes \mathbb{Z}[t^{\pm 1}]) \rtimes (G \times \mathbb{Z})$ by gluing $\tilde{f}_0$ and $\tilde{f}_J$.

By the definitions of $\tilde{f}_0$ and $\tilde{f}_J$, $\tilde{f}$ is a lift of $f \times \phi$ (remark that $\phi$ is zero on $G(J)$ since ${\rm lk}(K, \alpha) = 0$). We take $u_1 \in G(J)$ such that $f(u_1) = h'$. Let $u_2 \in G(K \cup \alpha)$ be a meridional loop of $K$ and $d_1, d_2$ the orders of $f(u_1 u_2), f(u_2) \in G$. Then, we have
\begin{eqnarray*}
\tilde{f}((u_1u_2)^{d_1 d_2} u_2^{-d_1 d_2}) &=& \bigl(((f(u_1),0) - (e, 0))v; f(u_1 u_2), 1\bigr)^{d_1 d_2} (0; e, -d_1 d_2)\\
&=& \left( \sum_{j=0}^{d_1 d_2 -1} (f(u_1u_2)^j, j) ((f(u_1),0) - (e, 0))v; e, 0 \right).
\end{eqnarray*} 
Since $f(u_1)v_0 = h'v_0 \neq v_0$, $\sum_{j=0}^{d_1 d_2 -1} (f(u_1u_2)^j, j) ((f(u_1),0) - (e, 0))v \neq 0$. Therefore $\tilde{f}$ is a nontrivial lift and we have $\Delta^{\rho \circ f}_{K(\alpha, J)}(t) = 0$ by Theorem \ref{thm_52}.
\end{proof}

The following two lemmas are likely well known to experts. 
However, for the sake of completeness, we include their proofs. 

\begin{lemma}\label{irr-lem}
Let $G$ be a finite group and $\rho \colon G \to {\rm GL}(V)$ be an irreducible representation. For any normal subgroup $N$ of $G$, the subspace $W = \{v \in V \mid \text{$nv = v$ for any $n \in N$}\}$ is either $\{0\}$ or $V$. In particular, if $\rho$ is faithful, $W = \{0\}$.
\end{lemma}
\begin{proof}
For any $g \in G, n \in N, v \in W$, $n(gv) = g(g^{-1} n g v) = gv,$ since $N$ is normal. Therefore we have $gW = W$ for any $g$, which means that $W$ is a $G$-invariant subspace. Since $V$ is assumed to be irreducible, $W$ is $\{0\}$ or $V$.
\end{proof}

Let $G$ be a finite non-nilpotent metabelian group and assume that every proper subgroup of $G$ is nilpotent. Then, the commutator subgroup $G'$ is the Sylow $q$-subgroup for some prime number $q$, and $G/G'$ is cyclic with order $p^h$ for some prime $p$. Let $P$ be a Sylow $p$-subgroup and $z \in P$ a generator; we find that the conjugacy action of $z$ on $G'$ has trivial fixed point and order $p$.

Let $\eta_k \colon G \to {\rm GL}(1, \mathbb{C})\;(k=0, \dots, p^h - 1)$ be the one-dimensional representation defined by $\eta_k(z) = \exp \frac{2\pi ik}{p^h}$ and $\eta_i|_{G'} = {\rm id}_{G'}$. We can easily find that these are the one-dimensional (irreducible) representations of $G$.

To construct the other representations, let $\mathcal{R}$ be the set of the nontrivial one-dimensional representations of $G'$, i.e., the set of the nontrivial homomorphisms from $G'$ to $\mathbb{C}^\times$. The quotient group $P/\langle z^p \rangle$ freely acts on $\mathcal{R}$, and each orbit of this action consists of exactly $p$ representations, i.e., $\rho, z^* \rho, \dots, (z^{p-1})^* \rho$ for a representation $\rho \in \mathcal{R}$. Denote $(z^i)^*\rho$ by $\rho_i \colon G' \to {\rm GL}(V_i)$. Then, we can define a $p$-dimensional representation of $G$ on $\bigoplus_{i=0}^{p-1} V_i$. By this procedure, we obtain representations $\tau_1, \dots, \tau_\ell$, where $\ell$ is the number of the orbits, i.e., $(|G'|-1)/p$.

\begin{lemma}\label{ma-lem}
The representations $\tau_1, \dots, \tau_\ell$ are irreducible. Furthermore, the following are the irreducible representations of $G$:
$$\eta_k\;(k=0,\dots, p^h-1),\quad \tau_j \otimes \eta_k\;(j=1, \dots, \ell; k=0,\dots, p^{h-1}-1).$$
\end{lemma}
\begin{proof}
Let $\tau_j$ be defined on $V = V_0 \oplus \cdots \oplus V_{p-1}$ as above. Because the restriction of $\tau_j$ to $G'$ is decomposed into subrepresentations on $V_0, \dots, V_{p-1}$, any subrepresentation $W \subset V$ of $G$ is described as the direct sum $\bigoplus_{i \in I} V_i$ for some subset $I \subset \{0, \dots, p-1\}$. However, the action of $P$ permutes representations $V_i$ transitively, and hence $I$ is $\emptyset$ or $\{0, \dots, p-1\}$, i.e., $W$ equals $\{0\}$ or $V$; $V$ is irreducible.

Because $\eta_k$ are one-dimensional, $\tau_j \otimes \eta_k$ are also irreducible. Furthermore, we can easily check that the representations $\eta_k\;(k=0,\dots, p^h-1),\; \tau_j \otimes \eta_k\;(j=1, \dots, \ell; k=0,\dots, p^{h-1}-1)$ are mutually inequivalent by comparing the characters. Since the square sum of the dimensions is
$$p^h \cdot 1^2 + \frac{|G'| - 1}{p} \cdot p^{h-1} \cdot p^2 = p^h |G'| = |G|,$$
this is the complete set of irreducible representations.
\end{proof}

\begin{proof}[Proof of Theorem \ref{thm_faithfulirrrep}]
Since $G$ is of weight one, there exists a pair $(K_0, f_0)$ of a knot $K_0$ and an epimorphism $f_0 \colon G(K_0) \to G$. In the following, we take a satellite knot $K_0(\alpha, J)$ as $K$ and show that the corresponding twisted Alexander polynomial is zero.

Firstly, we consider the case where the commutator subgroup $G'$ of $G$ is nilpotent. By the well known property of nilpotent groups and Theorem \ref{thm:main-4}, $G'$ is the direct product of two or more nontrivial Sylow subgroups. Let $p, q$ be distinct prime divisors of $|G'|$ and let $P, Q$ denote the Sylow $p$-subgroup and Sylow $q$-subgroup, respectively. We set $P_0 = \{g \in P \mid g \in Z(P), g^p = e\}$, where $Z(P)$ is the center of $P$, and define $Q_0 \subset Q$ in the same way. Since $P_0$ and $Q_0$ are characteristic in $G'$, they are characteristic also in $G$ and especially normal subgroups in $G$. Let $g_p \in P_0$ be a nontrivial element. Since $\rho$ is faithful, we can take an eigenvalue $\alpha \neq 1$ of $\rho(g_p)$, and let $V_\alpha$ denote the eigenspace. Because $g_p \in Z(G')$, $V_\alpha$ is $G'$-invariant and especially $Q_0$-invariant. By Lemma \ref{irr-lem}, there exists $g_q \in Q_0$ such that $\rho(g_q)|_{V_\alpha} \neq {\rm id}_{V_\alpha}$; the order of some eigenvalue  of $\rho (g_p g_q)$ is $pq$. Then, there exists a required pair $(K, f)$ by Lemma \ref{v1-lem}.

Secondly, let us assume that $G'$ is not nilpotent and take a minimal non-nilpotent subgroup $H$ of $G'$. A minimal non-nilpotent group is called a \textit{Schmidt group}, and a classification of Schmidt groups is given in \cite{BER05-1}; in particular, $w(H) = 1$. As in the classification, there are exactly two prime divisors $p, q$ of $|H|$, and for one of them, $q$, the Sylow $q$-subgroup $Q$ is equal to the commutator subgroup. Furthermore, $H$ is a (possibly trivial) central extension of a metabelian group. Let us assume that $|H| = p^a q^b$. We take a Sylow $p$-subgroup $P$ and a generator $z$ of $P$; $z^p$ is in the center of $H$.

If $H$ is metabelian, take an irreducible subrepresentation $\rho_W \colon H \to {\rm GL}(W)$ of $\rho|_H$ with $\rho_W|_{H'}$ nontrivial. The kernel of $\rho_W$ is contained in the center $Z(H)$ of $H$, which is generated by $z^p$. If ${\rm Ker}\, \rho_W = Z(H)$, $\rho(z)$ has $1$ as an eigenvalue by Lemma \ref{ma-lem} and then we obtain the proposition by Lemma \ref{v2-lem}. If ${\rm Ker}\, \rho_W \neq Z(H)$, for some positive integer $c$, $\rho_W(z^{p^c})$ has a primitive $p$-th root of unity as an eigenvalue. Since $z^{p^c} \in Z(H)$, the eigenspace is $H$-invariant and hence equal to $H$. We take any $h \in H' \backslash \{e\}$ to find that $\rho_W(hz^{p^c})$ has a primitive $pq$-th root of unity as an eigenvalue. By Lemma \ref{v1-lem}, we can find a required pair $(K, f)$.

If $H$ is not metabelian, take an irreducible subrepresentation $\rho_W \colon H \to {\rm GL}(W)$ of $\rho|_H$ with $\rho_W|_{H''}$ nontrivial. We take any $h \in H' \backslash H''$. Since $H'' = Z(H)$, $h$ acts on $W$ as the multiplication by a $q$-th root of unity. Since the generator $z$ of $P$ is a weight element of $H$, $\rho_W(z^{p^c})$ for some non-negative integer $c$ has a $p$-th root of unity as an eigenvalue. We find that $\rho_W(hz^{p^c})$ has a $pq$-th root of unity to conclude the proposition by Lemma \ref{v1-lem}.
\end{proof}

%%%%%%%%%%%%%%%%%%%%%%%%%%%%%%%%%%%%%%%
\subsection*{Acknowledgments}
The authors are supported in part by 
JSPS KAKENHI Grant Numbers JP23K20799, 
JP25K07012, and JP25K17252. 

%%%%%%%%%%%%%%%%%%%%%%%%%%%%%%%%%%%%%%%
\bibliographystyle{amsplain}

\end{document}